%% file: aff-lin-gp.tex
\documentclass{amsart}
\usepackage{amsfonts, amsbsy, amsmath, amssymb}

\usepackage{mmacells}

\newtheorem{thm}{Theorem}[section]
\newtheorem{lem}[thm]{Lemma}
\newtheorem{cor}[thm]{Corollary}
\newtheorem{prop}[thm]{Proposition}

\newtheorem{rmk}[thm]{Remark}

\newtheorem{ques}[thm]{Question}
\newtheorem{o-q}[thm]{Open Question}
\newtheorem{thm-con}[thm]{Theorem-Conjecture}
\numberwithin{equation}{section}

\theoremstyle{definition}

\allowdisplaybreaks

\newcommand{\f}{\Bbb F}

\begin{document}

\title[Affine Linear Groups]{Optimal Binary Constant Weight Codes and Affine Linear Groups over Finite Fields}

\author[Xiang-dong Hou]{Xiang-dong Hou}
\address{Department of Mathematics and Statistics,
University of South Florida, Tampa, FL 33620}
\email{xhou@usf.edu}



\begin{abstract}
Let $\text{AGL}(1,\f_q)$ be the affine linear group of dimension $1$ over a finite field $\f_q$. $\text{AGL}(1,\f_q)$ acts sharply 2-transitively on $\f_q$. Given $S<\text{AGL}(1,\f_q)$ and an integer $k$ with $1\le k\le q$, does there exist a subset $B\subset\f_q$ with $|B|=k$ such that $S=\text{AGL}(1,\f_q)_B$? ($\text{AGL}(1,\f_q)_B=\{\sigma\in\text{AGL}(1,\f_q):\sigma(B)=B\}$ is the stabilizer of $B$ in $\text{AGL}(1,\f_q)$.) We derive a sum that holds the answer to this question. This result determines all possible parameters of binary constant weight codes that are constructed from the action of $\text{AGL}(1,\f_q)$ on $\f_q$ to meet the Johnson bound. Consequently, the values of the function $A_2(n,d,w)$ are determined for many parameters, where $A_2(n,d,w)$ is the maximum number of codewords in a binary constant weight code of length $n$, weight $w$ and minimum distance $\ge d$.
\end{abstract}

\maketitle

\section{Introduction}

Let $\f_q$ denote the finite field with $q$ elements. A binary constant weight code of length $n$ is a subset $C\subset\f_2^n$ such that every codeword of $C$ has the same Hamming weight. Let $n,d,w$ be positive integers. By an $(n,d,w)_2$ code, we mean a binary constant weight code of length $n$, weight $w$, and minimum distance $\ge d$. (This is not a standard notation in coding theory but serves the purpose of the present paper conveniently.) Let $A_2(n,d,w)$ be the maximum number of codewords in an $(n,d,w)_2$ code. The values of the function $A_2(n,d,w)$ have been the focus of numerous studies in coding theory and combinatorial designs; see for example \cite{Agrell-Vardy-Zeger-IEEE-IT-2000, Brouwer-Etzion, Brouwer-Shearer-Sloane-Smith-IEEE-IT-1990, Chee-Xing-Yeo-IEEE-IT-2010, Johnson-IEEE-IT-1962, Ostergard-IEEE-IT-2010, Schrijver-IEEE-IT-2005, Smith-Hughes-Perkins-E-JC-2006, Sun-JAA-2017, Tonchev-1988}. A fundamental fact about the function $A_2(n,d,w)$ is the following upper bound.

\begin{thm}[The restricted Johnson bound {\cite{Johnson-IEEE-IT-1962}, \cite[\S2.3.1]{Huffman-Pless-2003}}]\label{T0} 
We have
\begin{equation}\label{1.1}
A_2(n,d,w)\le \frac{nd}{2w^2-2nw+nd}
\end{equation}
provided $2w^2-2nw+nd>0$.
\end{thm}

When $d=2\delta$, \eqref{1.1} gives
\begin{equation}\label{1.2}
A_2(n,2\delta,w)\le \frac{n\delta}{w^2-nw+n\delta}
\end{equation}
provided $w^2-nw+n\delta>0$. The codes achieving the equality in \eqref{1.2} are precisely those constructed from balanced incomplete block designs.

A {\em balanced incomplete block design} (BIBD) with parameters $(v,b,r,k,\lambda)$, where the parameters are positive integers, is a pair $D=(V,\mathcal B)$, where $V=\{x_1,\dots,x_v\}$ is a $v$-element set and $\mathcal B=\{B_1,\dots,B_b\}$ is a family of $b$ $k$-element subsets of $V$, and $B_1,\dots,B_b$, called {\em blocks}, are such that each element of $V$ belongs to $r$ blocks and any two elements of $V$ belong to $\lambda$ blocks. The incidence matrix of $D$ is the $v\times b$ matrix $A=(a_{ij})$ where
\[
a_{ij}=\begin{cases}
1&\text{if}\ x_i\in B_j,\cr
0&\text{otherwise}.
\end{cases}
\]
The obvious relations $bk=vr$ and $r(k-1)=\lambda(v-1)$ allow $r$ and $b$ to be expressed in terms of $v,k,\lambda$:
\begin{equation}\label{1.2.0}
r=\frac{\lambda(v-1)}{k-1},\qquad b=\frac{\lambda v(v-1)}{k(k-1)}.
\end{equation}
Hence a $(v,b,r,k,\lambda)$ BIBD is also called a $(v,k,\lambda)$ BIBD.
We treat the $(0,1)$-matrix $A$ as being over $\f_2$ and let $R(A)$ denote the set of rows of $A$. Then $R(A)$ is a $(n,2\delta, w)_2$ code with minimum distance $2\delta$, where
\begin{equation}\label{1.3}
\begin{cases}
n=b=\displaystyle \frac{\lambda v(v-1)}{k(k-1)},\vspace{2mm}\cr
\delta=r-\lambda=\displaystyle \frac{\lambda(v-k)}{k-1},\vspace{2mm}\cr
w=r=\displaystyle \frac{\lambda(v-1)}{k-1}.
\end{cases}
\end{equation}

\begin{thm}[{\cite{Semakov-Zinov'ev-PIT-1969}, \cite[Theorem~2.4.12]{Tonchev-1988}}]\label{T1.1}
$C$ is an $(n,2\delta,w)_2$ code with $|C|=n\delta/(w^2-nw+n\delta)$ if and only if $C=R(D)$, where $D$ is the incidence matrix of a $(v,b,r,k,\lambda)$ BIBD whose parameters satisfy \eqref{1.3}. Hence the equality in \eqref{1.2} holds if and only if there exists a $(v,b,r,k,\lambda)$ BIBD whose parameters satisfy \eqref{1.3}.
\end{thm}

Let $G$ be a finite group acting 2-transitively on a finite set $V$ and let $B\subset V$ be such that $|B|\ge 2$. Let $\mathcal B=\{g(B):g\in G\}$ and let 
\[
G_B=\{g\in G:g(B)=B\}<G
\]
be the stabilizer of $B$ in $G$. It is quite obvious that $(V,\mathcal B)$ is a BIBD; its parameters are given by 
\begin{equation}\label{1.4}
\begin{cases}
v=|V|,\vspace{2mm}\cr
b=\displaystyle \frac{|G|}{|G_B|},\vspace{2mm}\cr
r=\displaystyle \frac{|B|}{|V|}\cdot\frac{|G|}{|G_B|},\vspace{2mm}\cr
k=|B|,\vspace{2mm}\cr
\lambda=\displaystyle \frac{|B|(|B|-1)}{|V|(|V|-1)}\cdot\frac{|G|}{|G_B|}.
\end{cases}
\end{equation}
With $G$ and $V$ given, the parameters of $(V,\mathcal B)$ are determined by $|B|$ and $|G_B|$. Therefore, an answer to the following question would determine all possible parameters of $(V,\mathcal B)$.

\begin{ques}\label{Q1.2}
Let $G$ be a finite groups acting on a finite set $V$. Let $S<G$ and $0\le k\le |V|$. Does there exist $B\subset V$ with $|B|=k$ such that $S=G_B$?
\end{ques}

The affine linear group $\text{AGL}(1,\f_q)$ acts (sharply) 2-transitively on $\f_q$. The main objective of the present paper is to answer Question~\ref{Q1.2} for this action. More precisely, we determine all subgroups of $\text{AGL}(1,\f_q)$, and for each $S<\text{AGL}(1,\f_q)$ and each $0\le k\le q$, we derive a sum that gives the number of $k$-element subsets $B$ of $\f_q$ such that $S=\text{AGL}(1,\f_q)_B$. Consequently, we are able to describe the possible parameters of the BIBDs constructed from the action of $\text{AGL}(1,\f_q)$ on $\f_q$. By Theorem~\ref{T1.1}, the corresponding values of the function $A_2(n,2\delta,w)$ are determined. More precisely, if there exist integers $k\ge 2$ and $s>0$ such that there are $S<\text{AGL}(a,\f_q)$ and $B\in\binom{\f_q}k$ with $|S|=s$ and $S=\text{AGL}(1,\f_q)_B$, then there is a $(q,k,k(k-1)/s)$ BIBD and hence 
\begin{equation}\label{1.6}
A_2\Bigl(\frac{q(q-1)}s,\,\frac{2k(q-k)}s,\,\frac{k(q-1)}s\Bigr)=q.
\end{equation}

Sun \cite{Sun-TJM-2010} constructed various subsets of $\f_q$ and determined their stabilizers in $\text{AGL}(1,\f_q)$. The approach of the present paper is a little different; we are interested in the existence of $B\in\binom{\f_q}k$, not its description, such that $\text{AGL}(1,\f_q)_B$ is a given subgroup of $\text{AGL}(1,\f_q)$. 

Question~\ref{Q1.2} has been considered for $\text{PGL}(2,\f_q)$ and $\text{PSL}(2,\f_q)$ acting on the projective line $\text{PG}(1,\f_q)$ by Cameron, Omidi and Tayfeh-Rezaie \cite{Cameron-Omidi-Tayfeh-Rezaie-E-JC-2006} and by Cameron, Maimani, Omidi and Tayfeh-Rezaie \cite{Cameron-Maimani-Omidi-Tayfeh-Rezaie-DM-2006}. The purpose there is to determine possible parameters of $3$-designs admitting $\text{PGL}(2,\f_q)$ or $\text{PSL}(2,\f_q)$ as an automorphism group. 

Here is the outline of the paper. In Section~2, we outline a strategy for solving Question~\ref{Q1.2}. Section~3 provides the necessary background on $\text{AGL}(1,\f_q)$ and its subgroups. Section~4 contains the main results of the paper. Theorems~\ref{T4.2} and \ref{T4.3} give formulas for the function $\mathcal N(S,k)=|\{B\in\binom{\f_q}k:S=\text{AGL}(1,\f_q)_B\}|$, where $S<\text{AGL}(1,\f_q)$ and $0\le k\le q$. In Section~5, one finds a description of the input data of the aforementioned formulas. Section~6 contains some additional remarks about Question~\ref{Q1.2} and our formulas. It is observed that our numerical results are consistent with the existing theoretic results. The section also includes several open questions. The appendix supplies the reader with a Mathematica code for computing the function $\mathcal N(S,k)$ and a table of values of $\mathcal N(S,k)$ with $q\le 101$. ({\bf Caution:} The table is very long.)

\section{A Strategy for Solving Question~\ref{Q1.2}}

Let $G$ be a finite group acting on a finite set $V$. Let $S<G$, $0\le k\le |V|$, and $\binom Vk=\{B\subset V:|B|=k\}$. Define
\begin{equation}\label{2.1}
S'=\Bigl\{B\in\binom Vk:\sigma(B)=B\ \text{for all}\ \sigma\in S\Bigr\}.
\end{equation}
(We remind the reader that the operation $(\ )'$ depends on $k$.) Note that a subset $B$ of $V$ belongs to $S'$ if and only if $B$ is a union of $S$-orbits in $V$ with $|B|=k$. If the sizes of the $S$-orbits are known, then $|S'|$ is easily determined. Let $S_1,\dots,S_t$ be the minimal elements of $\{T:S\lneq T<G\}$; these are immediate supergroups of $S$ in $G$. Then by the inclusion-exclusion principle, 
\begin{align}\label{2.2}
&\Bigl|\Bigl\{B\in\binom Vk:S=G_B\Bigr\}\Bigr|\\
=\,&|S'|-|S_1'\cup\cdots\cup S_t'|\cr
=\,&|S'|-\sum_{l=1}^t(-1)^{l-1}\sum_{1\le i_1<\cdots<i_l\le t}|S_{i_1}'\cap\cdots\cap S_{i_l}'|\cr
=\,&|S'|+\sum_{l=1}^t(-1)^{l}\sum_{1\le i_1<\cdots<i_l\le t}|\langle S_{i_1}\cup\cdots\cup S_{i_l}\rangle'|\cr
=\,&\sum_{l=0}^t(-1)^l\sum_{1\le i_1<\cdots<i_l\le t}|\langle S\cup S_{i_1}\cup\cdots\cup S_{i_l}\rangle'|,\nonumber
\end{align}
where $\langle\ \rangle$ denotes the subgroup generated by a subset. The answer to Question~\ref{Q1.2} is affirmative if and only if the sum in \eqref{2.2} is nonzero. However, to make this sum computable (for all $S$ and $k$), one needs the following data: (i) an enumeration of all subgroups of $G$ and a formula for $|S'|$ for each $S<G$; (ii) for each $S<G$, an enumeration of all immediate supergroups of $S$; (iii) a description of the subgroup generated by any set of immediate supergroups of a subgroup $S$. In the next two sections, we gather these data for $\text{AGL}(1,\f_q)$ acting on $\f_q$.

\section{The Affine Linear Group $\text{AGL}(1,\f_q)$}

\subsection{$\text{AGL}(1,\f_q)$}\

The affine linear group of dimension 1 over $\f_q$ is defined as 
\begin{equation}\label{3.1}
\text{AGL}(1,\f_q)=\Bigl\{ \left[\begin{matrix} a&b\cr 0&1\end{matrix}\right]:a\in\f_q^*,\ b\in\f_q\Bigr\}<\text{GL}(2,\f_q).
\end{equation}
The action of $\text{AGL}(1,\f_q)$ on $\f_q$ is as follows: For $\left[\begin{smallmatrix} a&b\cr 0&1\end{smallmatrix}\right]\in\text{AGL}(1,\f_q)$ and $x\in\f_q$,
\begin{equation}\label{3.2}
\left[\begin{matrix} a&b\cr 0&1\end{matrix}\right]x=ax+b.
\end{equation}
This is clearly a sharply 2-transitive action. Let
\begin{equation}\label{3.3}
A=\Bigl\{ \left[\begin{matrix} a&0\cr 0&1\end{matrix}\right]:a\in\f_q^*\Bigr\}\cong\f_q^*,
\end{equation}
\begin{equation}\label{3.4}
B=\Bigl\{ \left[\begin{matrix} 1&b\cr 0&1\end{matrix}\right]:b\in\f_q\Bigr\}\cong\f_q.
\end{equation}
Then we have
\begin{equation}\label{3.5}
\text{AGL}(1,\f_q)=A\ltimes B.
\end{equation}
We list a few formulas that are frequently used as we proceed:
\begin{equation}\label{3.6}
\left[\begin{matrix} a&b\cr 0&1\end{matrix}\right]^{-1}=\left[\begin{matrix} a^{-1}&-a^{-1}b\cr 0&1\end{matrix}\right],
\end{equation}
\begin{equation}\label{3.7}
\left[\begin{matrix} a&b\cr 0&1\end{matrix}\right]^l=\left[\begin{matrix} a^l&\displaystyle\frac{a^l-1}{a-1}b\vspace{2mm}\cr 0&1\end{matrix}\right],\quad l\in\Bbb Z,\quad \text{($\frac{a^l-1}{a-1}$ interpreted as $l$ for $a=1$)},
\end{equation}
\begin{equation}\label{3.8}
\left[\begin{matrix} a_1&(a_1-1)b\cr 0&1\end{matrix}\right]\left[\begin{matrix} a_2&(a_2-1)b\cr 0&1\end{matrix}\right]=\left[\begin{matrix} a_1a_2&(a_1a_2-1)b\cr 0&1\end{matrix}\right].
\end{equation}

\subsection{Subgroups of $\text{AGL}(1,\f_q)$}\label{s3.2}\

For each $H<\f_q$, define
\begin{equation}\label{3.9}
\overline H=\Bigl\{ \left[\begin{matrix} 1&h\cr 0&1\end{matrix}\right]:h\in H\Bigr\}.
\end{equation}
Let $p=\text{char}\,\f_q$. Let $\mathcal S$ be the set of triples $(a,b,H)$, where $a\in\f_q^*$, $H$ is an $\f_p(a)$-subspace of $\f_q$, and
\[
b
\begin{cases}
=0&\text{if}\ a=1,\cr
\in\f_q&\text{if}\ a\ne 1.
\end{cases}
\]
For $(a,b,H)\in\mathcal S$, let 
\begin{equation}\label{3.10}
S(a,b,H)=\Bigl\langle\Bigl\{ \left[\begin{matrix} a&b\cr 0&1\end{matrix}\right]\Bigr\}\cup\overline H\Bigr\rangle<\text{AGL}(1,\f_q).
\end{equation}
Note that for $h\in H$,
\begin{equation}\label{3.10.0}
\left[\begin{matrix} a&b\cr 0&1\end{matrix}\right]\left[\begin{matrix} 1&h\cr 0&1\end{matrix}\right]\left[\begin{matrix} a&b\cr 0&1\end{matrix}\right]^{-1}=\left[\begin{matrix} 1&ah\cr 0&1\end{matrix}\right]\in\overline H
\end{equation}
since $ah\in H$. Moreover, by \eqref{3.7},
\[
\Bigl\langle\left[\begin{matrix} a&b\cr 0&1\end{matrix}\right]\Bigr\rangle\cap\overline H=\Bigl\{\left[\begin{matrix} 1&0\cr 0&1\end{matrix}\right]\Bigr\}.
\]
Hence
\begin{equation}\label{3.11}
S(a,b,H)=\Bigl\langle\left[\begin{matrix} a&b\cr 0&1\end{matrix}\right]\Bigr\rangle\ltimes\overline H.
\end{equation}
\eqref{3.7} also implies that
\begin{equation}\label{3.12}
o\Bigl(\left[\begin{matrix} a&b\cr 0&1\end{matrix}\right]\Bigr)=o(a),
\end{equation}
where $o(\ )$ denotes the order of an element in a group. Thus
\begin{equation}\label{3.13}
|S(a,b,H)|=o(a)|H|.
\end{equation}

\begin{thm}\label{T3.1}
Every subgroup of $\text{\rm AGL}(1,\f_q)$ is of the form $S(a,b,H)$ for some $(a,b,H)\in\mathcal S$.
\end{thm}

\begin{proof}
Let $S<\text{AGL}(1,\f_q)$. Let $\phi:S\to\text{AGL}(1,\f_q)/B$ be the composition of the following homomorphisms
\[
S\hookrightarrow \text{AGL}(1,\f_q)\longrightarrow \text{AGL}(1,\f_q)/B\cong\f_q^*,
\]
where $B$ is given in \eqref{3.4}. Then $\phi$ induces an embedding 
\begin{equation}\label{3.14}
S/\ker \phi\hookrightarrow \text{AGL}(1,\f_q)/B\cong\f_q^*.
\end{equation}
Since $\ker\phi\subset B$, we have $\ker\phi=\overline H$ for some $H<\f_q$. By \eqref{3.14}, $S/\overline H$ is cyclic with $|S/\overline H|\mid q-1$. Write
\begin{equation}\label{3.15}
S/\overline H=\Bigl\langle\left[\begin{matrix} a&b\cr 0&1\end{matrix}\right]\Bigr\rangle, \qquad\text{where}\ \left[\begin{matrix} a&b\cr 0&1\end{matrix}\right]\in S.
\end{equation}
Since 
\[
\left[\begin{matrix} a&b\cr 0&1\end{matrix}\right]\overline H\left[\begin{matrix} a&b\cr 0&1\end{matrix}\right]^{-1}=\overline H.
\]
It follows from \eqref{3.10.0} that $aH=H$. Hence $H$ is an $\f_p(a)$-subspace of $\f_q$. If $a=1$, by \eqref{3.15}, $S/\overline H$ is a $p$-group. Then $|S/\overline H|=1$, and hence $S=\overline H=S(1,0,H)$, where $(1,0,H)\in\mathcal S$. If $a\ne 1$, then $(a,b,H)\in\mathcal S$ and 
\[
S=\Bigl\langle\Bigl\{ \left[\begin{matrix} a&b\cr 0&1\end{matrix}\right]\Bigr\}\cup \overline H\Bigr\rangle=S(a,b,H).
\]
\end{proof}

We fix a generator $\gamma$ of $\f_q^*$ and let $\Gamma=\{\gamma^i: i\mid q-1\}$. Each subgroup of $\f_q^*$ has a unique generator in $\Gamma$.

\begin{prop}\label{P3.2}
\begin{itemize}
\item[(i)] For each $(a,b,H)\in\mathcal S$, there exists $(a',b',H)\in\mathcal S$ with $a'\in\Gamma$ such that $S(a,b,H)=S(a',b',H)$.

\item[(ii)] Let $(a_1,b_1,H_1),(a_2,b_2,H_2)\in\mathcal S$ be such that $a_1,a_2\in\Gamma$. Then $S(a_1,b_1,H_1)=S(a_2,b_2,H_2)$ if and only if $H_1=H_2$, $a_1=a_2$, and $b_1\equiv b_2\pmod{H_1}$.

\item[(iii)] For each $(a,b,H)\in\mathcal S$, there exists $c\in\f_q$ such that 
\begin{equation}\label{3.16}
\left[\begin{matrix} 1&c\cr 0&1\end{matrix}\right]S(a,b,H)\left[\begin{matrix} 1&c\cr 0&1\end{matrix}\right]^{-1}=S(a,0,H).
\end{equation}
\end{itemize}
\end{prop}

\begin{proof}
(i) There exists $a'\in\Gamma$ such that $\langle a\rangle=\langle a'\rangle$, that is, $a'=a^l$ for some $l\in\Bbb Z$ with $\text{gcd}(l,o(a))=1$. Write 
\[
\left[\begin{matrix} a&b\cr 0&1\end{matrix}\right]^l=\left[\begin{matrix} a'&b'\cr 0&1\end{matrix}\right].
\]
Then clearly, $(a',b',H)\in\mathcal S$ and $S(a,b,H)=S(a',b',H)$.

\medskip
(ii) ($\Leftarrow$) Since 
\[
\left[\begin{matrix} a_1&b_1\cr 0&1\end{matrix}\right]^{-1}\left[\begin{matrix} a_1&b_2\cr 0&1\end{matrix}\right]=\left[\begin{matrix} 1&(b_2-b_1)/a_1\cr 0&1\end{matrix}\right]\in\overline H,
\]
it is clear that $S(a_1,b_1,H)=S(a_1,b_2,H)$.

\medskip
($\Rightarrow$) For $i=1,2$, $\overline H_i$ is the unique Sylow $p$-subgroup of $S(a_i,b_i,H_i)$. Since $S(a_1,b_1,H_1)=S(a_2,b_2,H_2)$, we have $\overline H_1=\overline H_2$ and hence $H_1=H_2$. It follows from $|S(a_1,b_1,H_1)|=|S(a_2,b_2,H_1)|$ that $o(a_1)=o(a_2)$. Since $a_1,a_2\in\Gamma$, we must have $a_1=a_2$. Since 
\[
\Bigl\langle\left[\begin{matrix} a_1&b_1\cr 0&1\end{matrix}\right]\overline H_1\Bigr\rangle=S(a_1,b_1,H_1)/\overline H_1= S(a_1,b_2,H_1)/\overline H_1=\Bigl\langle\left[\begin{matrix} a_1&b_2\cr 0&1\end{matrix}\right]\overline H_1\Bigr\rangle,
\]
there exists $0<l\le o(a_1)$ such that 
\[
\left[\begin{matrix} a_1&b_1\cr 0&1\end{matrix}\right]^l\overline H_1=\left[\begin{matrix} a_1&b_2\cr 0&1\end{matrix}\right]\overline H_1.
\]
It follows that $a_1^l=a_1$ and hence $l=1$. Now from 
\[
\left[\begin{matrix} a_1&b_1\cr 0&1\end{matrix}\right]\overline H_1=\left[\begin{matrix} a_1&b_2\cr 0&1\end{matrix}\right]\overline H_1
\]
we have $b_1\equiv b_2\pmod{H_1}$.

\medskip
(iii) Let
\[
c=\begin{cases}
0&\text{if}\ a=1,\vspace{2mm}\cr
\displaystyle\frac b{a-1}&\text{if}\ a\ne 1.
\end{cases}
\]
Then
\[
\left[\begin{matrix} 1&c\cr 0&1\end{matrix}\right]\left[\begin{matrix} a&b\cr 0&1\end{matrix}\right]\left[\begin{matrix} 1&c\cr 0&1\end{matrix}\right]^{-1}=\left[\begin{matrix} a&(1-a)c+b\cr 0&1\end{matrix}\right]=\left[\begin{matrix} a&0\cr 0&1\end{matrix}\right].
\]
Of course,
\[
\left[\begin{matrix} 1&c\cr 0&1\end{matrix}\right]\overline H\left[\begin{matrix} 1&c\cr 0&1\end{matrix}\right]^{-1}=\overline H.
\]
Hence \eqref{3.16} holds.
\end{proof}

\begin{prop}\label{P3.3}
Let $(a,0,H)\in\mathcal S$, $a\in\Gamma$.
\begin{itemize}
\item[(i)]
If $a=1$, the subgroups of $\text{\rm AGL}(1,\f_q)$ that contain $S(1,0,H)=\overline H$ are precisely $S(a_1,b_1,H_1)$, where $(a_1,b_1,H_1)\in\mathcal S$, $a_1\in\Gamma$, $H\subset H_1$.

\item[(ii)]
If $a\ne 1$, the subgroups of $\text{\rm AGL}(1,\f_q)$ that contain $S(a,0,H)$ are precisely $S(a_1,0,H_1)$, where $(a_1,0,H_1)\in\mathcal S$, $a_1\in\Gamma$, $a\in\langle a_1\rangle$, $H\subset H_1$.
\end{itemize}
\end{prop}

\begin{proof}
(i) is obvious.

\medskip
(ii) First assume that $(a_1,0,H_1)\in\mathcal S$ is such that $a\in\langle a_1\rangle$ and $H\subset H_1$. Then $\overline H\subset\overline H_1$ and $\left[\begin{smallmatrix} a&0\cr 0&1\end{smallmatrix}\right]\in\langle \left[\begin{smallmatrix} a_1&0\cr 0&1\end{smallmatrix}\right]\rangle$. Hence $S(a,0,H)\subset S(a_1,0,H_1)$.

Now assume that $S(a,0,H)\subset S(a_1,b_1,H_1)$, where $(a_1,b_1,H_1)\in\mathcal S$, $a_1\in\Gamma$. Since $\overline H$ is a $p$-subgroup of $S(a_1,b_1,H_1)$, we have $\overline H\subset\overline H_1$, and hence $H\subset H_1$. Since $\left[\begin{smallmatrix} a&0\cr 0&1\end{smallmatrix}\right]\in S(a_1,b_1,H_1)$, there exists integer $l>0$ such that
\begin{equation}\label{3.17}
\left[\begin{matrix} a&0\cr 0&1\end{matrix}\right]\in \left[\begin{matrix} a_1&b_1\cr 0&1\end{matrix}\right]^l\overline H_1.
\end{equation}
Then $a_1^l=a$ and \eqref{3.17} becomes
\[
\left[\begin{matrix} a&0\cr 0&1\end{matrix}\right]\in\left[\begin{matrix} a&b_1(a-1)/(a_1-1)\cr 0&1\end{matrix}\right]\overline H_1.
\]
Thus $b_1(a-1)/(a_1-1)\in H_1$, i.e., $b_1\in H_1$. We then have $S(a_1,b_1,H_1)=S(a_1,0,H_1)$.
\end{proof}

\subsection{Orbits of $S(a,b,H)$}\label{s3.3}\

Let $(a,b,H)\in\mathcal S$. If $a=1$, the orbits of $S(1,0,H)=\overline H$ in $\f_q$ are clearly the cosets of $H$ in $\f_q$.

Now assume that $a\ne 1$. Let $\sigma=\left[\begin{smallmatrix} a&b\cr 0&1\end{smallmatrix}\right]$. If $x,y\in\f_q$ are such that $x\equiv y\pmod H$, then $\sigma(x)-\sigma(y)=a(x-y)\equiv 0\pmod H$. Moreover, for all $\beta\in\overline H$ and $x\in\f_q$, $\beta x\equiv x\pmod H$. Hence the action of $S(a,b,H)$ on $\f_q$ induces an action of $S(a,b,H)/\overline H\cong\langle\sigma\rangle$ on $\f_q/H$. For all $0\le l<o(a)$ and $x\in\f_q$,
\[
\sigma^lx=a^lx+b(a^l-1)/(a-1).
\]
Note that $\sigma^lx\equiv x\pmod H$ if and only if $x\equiv-b/(a-1)\pmod H$. Therefore, the $\langle\sigma\rangle$-orbits in $\f_q/H$ containing $x+H$ is 
\[
\begin{cases}
\{x+H\}&\text{if}\ x+H=-b/(a-1)+H,\cr
\{a^lx+b(a^l-1)/(a-1)+H: 0\le l<o(a)\}&\text{otherwise}.
\end{cases}
\]
Consequently, the $S(a,b,H)$-orbit in $\f_q$ containing $x$ is
\[
\begin{cases}
x+H&\text{if}\ x\equiv-b/(a-1)\pmod H,\cr
\bigcup_{0\le l<o(a)}\bigl(a^lx+b(a^l-1)/(a-1)+H\bigr)&\text{otherwise}.
\end{cases}
\]
Note that $S(a,b,H)$ has one orbit of size $|H|$ and $(q-|H|)/o(a)|H|$ orbits of size $o(a)|H|$ in $\f_q$.

\begin{cor}\label{C3.4}
Let $0\le k\le q$ be fixed and $(\ )'$ be defined in \eqref{2.1}. Let $(a,b,H)\in\mathcal S$.
\begin{itemize}
\item[(i)]
If $a=1$, 
\begin{equation}\label{3.18}
|S(1,0,H)'|=
\begin{cases}
\displaystyle \binom{q/|H|}{k/|H|}&\text{if}\ k\equiv 0\pmod{|H|},\vspace{2mm}\cr
0&\text{otherwise}.
\end{cases}
\end{equation}

\item[(ii)] 
If $a\ne 1$,
\begin{equation}\label{3.19}
|S(a,b,H)'|=
\begin{cases}
\displaystyle \binom{(q-|H|)/o(a)|H|}{k/o(a)|H|}&\text{if}\ k\equiv 0\pmod{o(a)|H|},\vspace{2mm}\cr
\displaystyle \binom{(q-|H|)/o(a)|H|}{(k-|H|)/o(a)|H|}&\text{if}\ k\equiv |H|\pmod{o(a)|H|},\vspace{2mm}\cr
0&\text{otherwise}.
\end{cases}
\end{equation}
\end{itemize}
\end{cor}
 For integers $u,v>0$, define
\begin{equation}\label{3.20}
s_{q,k}(u,v)=\binom{(q-v)/uv}{k/uv}+\binom{(q-v)/uv}{(k-v)/uv},
\end{equation}
where a binomial coefficient $\binom xy$ is defined to be $0$ if $y\notin\Bbb N$. Then \eqref{3.18} and \eqref{3.19} can be combined as 
\begin{equation}\label{3.21}
|S(a,b,H)'|=s_{q,k}(o(a),|H|)\qquad\text{for all}\ (a,b,H)\in\mathcal S.
\end{equation}

\begin{rmk}\label{R3.5}
$s_{q,k}(u,v)=0$ unless $k\equiv 0$ or $v\pmod{uv}$. 
\end{rmk}


\section{A Solution to Question~\ref{Q1.2} for $\text{AGL}(1,\f_q)$}\label{s4}

First, we need to extend our notation. Let $q=p^\alpha$. For a vector space $V$ over a field $K$, let $\Bbb P_K(V)$ be the set of all $1$-dimensional subspaces of $V$. For each $H<\f_q$, let $H'=\{x\in\f_q:xH\subset H\}$, which is the largest subfield $K$ of $\f_q$ such that $H$ is a $K$-module. For each subset $X\subset\f_q$ and subfield $K\subset\f_q$, let $\langle X\rangle^K$ denote the $K$-span of $X$. For each positive integer $u$, let $\mathcal P(u)$ denote the set of all prime divisors of $u$; if $v\in\Bbb Z$ is such that $\text{gcd}(u,v)=1$, let $o_u(v)$ denote the multiplicative order of $v$ modulo $u$. If $P\subset\Bbb Z$ is finite, define $\Pi P=\prod_{e\in P}e$. Recall that $\gamma$ is a fixed generator of $\f_q^*$.

Let $\mu(\cdot,\cdot)$ denote the M\"obius function of the partially ordered set of all finite dimensional subspaces of a vector space over $\f_q$. It is well known \cite{Bender-Goldman-AMM-1975} that if $W\subset V$ are finite dimensional vector spaces over $\f_q$, then
\begin{equation}\label{4.1.0}
\mu_q(W,V)=(-1)^{\dim_{\f_q}(V/W)}q^{\binom{\dim_{\f_q}(V/W)}2}.
\end{equation}
Recall that the number of $m$-dimensional $\f_q$-subspaces of an $n$-dimensional $\f_q$-vector space is given by the $q$-binomial coefficient 
\begin{equation}\label{4.1.1}
{n\brack m}_q=\prod_{i=0}^{m-1}\frac{1-q^{n-i}}{1-q^{i+1}}.
\end{equation}

Given $S<\text{AGL}(1,\f_q)$ and $0\le k\le q$, let 
\begin{equation}\label{4.1}
\mathcal N(S,k)=\Bigl|\Bigl\{B\in\binom{\f_q}k:S=\text{AGL}(1,\f_q)_B\Bigr\}\Bigr|.
\end{equation}
Obviously, 
\begin{equation}\label{4.3.0}
\mathcal N(S,k)\le |S'|,
\end{equation}
but the equality does not hold in general.
The objective of this section is to determine $\mathcal N(S,k)$. By Proposition~\ref{P3.2} (i) and (iii), it suffices to consider $S=S(\gamma^{(q-1)/d},0,H)$, where $d\mid q-1$ and $(\gamma^{(q-1)/d},0,H)\in\mathcal S$. Since $\text{AGL}(1,\f_q)_B=\text{AGL}(1,\f_q)_{\f_q\setminus B}$, we also have
\begin{equation}\label{4.4.0}
\mathcal N(S,k)=\mathcal N(S,q-k).
\end{equation}
Hence it suffices to consider $0\le k\le q/2$.

\subsection{$S=S(1,0,H)$}\label{s4.1}\

Consider $S=S(1,0,H)$, where $(1,0,H)\in\mathcal S$. Let $\phi:\f_q\to\f_q/H$ be the canonical homomorphism and let $\mathcal B$ be a system of coset representatives of $H$ in $\f_q$. By Proposition~\ref{P3.3}, the immediate supergroups of $S(1,0,H)$ in $\text{AGL}(1,\f_q)$ are precisely 
\begin{equation}\label{4.2}
S(\gamma^{(q-1)/e},b,H),\qquad e\in\mathcal P(|H'|-1),\ b\in\mathcal B,
\end{equation}
and 
\begin{equation}\label{4.3}
S(1,0,\phi^{-1}(x)),\qquad x\in\Bbb P_{\f_p}(\f_q/H).
\end{equation}
Let $C\subset \mathcal P(|H'|-1)\times\mathcal B$ and $X\subset\Bbb P_{\f_p}(\f_q/H)$ be arbitrary. Define
\begin{equation}\label{4.4}
B(C)=\Bigl\{\frac b{\gamma^{(q-1)/e}-1}: (e,b)\in C\Bigr\},
\end{equation}
\begin{equation}\label{4.5}
\Delta(C)=\{u-v:u,v\in B(C)\}.
\end{equation}
Write
\begin{equation}\label{4.6}
C=\bigcup_{e\in P}(\{e\}\times B_e),
\end{equation}
where $P\subset\mathcal P(|H'|-1)$ and $\emptyset\ne B_e\subset\mathcal B$ for all $e\in P$.

\begin{lem}\label{L4.1}
In the above notation we have
\begin{align}\label{4.7}
&\Bigl\langle S(1,0,H)\cup\Bigl(\bigcup_{(e,b)\in C}S(\gamma^{(q-1)/e},b,H)\Bigr)\cup\Bigl(\bigcup_{x\in X}S(1,0,\phi^{-1}(x))\Bigr)\Bigr\rangle\\
=\,&S\Bigl(\gamma^{(q-1)/\Pi P},\, (\gamma^{(q-1)/\Pi P}-1)u,\, \Bigl\langle\Delta(C)\cup\phi^{-1}\Bigl(\Bigl\langle\bigcup_{x\in X}x\Bigr\rangle\Bigr)\Bigr\rangle^{\f_{p^{o_{\Pi P}(p)}}}\Bigr),\nonumber
\end{align}
where $u\in B(C)$ is arbitrary. (Note. When $C=\emptyset$, $\gamma^{(q-1)/\Pi P}-1=0$ and $u$ is irrelevant.)
\end{lem}

\begin{proof}
Let $L$ and $R$ denote the left and right side of \eqref{4.7}, respectively. Note that $\f_p(\gamma^{(q-1)/\Pi P})=\f_{p^{o_{\Pi P}(p)}}$; hence the triple of parameters in $R$ belongs to $\mathcal S$.

\medskip
$1^\circ$ We first show that $L\subset R$. Let 
\[
H_1=\Bigl\langle\Delta(C)\cup\phi^{-1}\Bigl(\Bigl\langle\bigcup_{x\in X}x\Bigr\rangle\Bigr)\Bigr\rangle^{\f_{p^{o_{\Pi P}(p)}}}.
\]
For each $x\in X$, clearly, $S(1,0,\phi^{-1}(x))\subset\overline H_1\subset R$. Moreover, for each $(e,b)\in C$, we have
\begin{align*}
R\,&\ni\left[\begin{matrix}\gamma^{(q-1)/\Pi P}&(\gamma^{(q-1)/\Pi P}-1)u\cr 0&1\end{matrix}\right]^{(\Pi P)/e}\cr
&=\left[\begin{matrix}\gamma^{(q-1)/e}&\frac{\gamma^{(q-1)/e}-1}{\gamma^{(q-1)/\Pi P}-1}(\gamma^{(q-1)/\Pi P}-1)u\cr 0&1\end{matrix}\right]\kern 5mm \text{(by \eqref{3.7})}\cr
&=\left[\begin{matrix}\gamma^{(q-1)/e}&(\gamma^{(q-1)/e}-1)u\cr 0&1\end{matrix}\right]\cr
&\in \left[\begin{matrix}\gamma^{(q-1)/e}&(\gamma^{(q-1)/e}-1)\frac b{\gamma^{(q-1)/e}-1}\cr 0&1\end{matrix}\right]\overline H_1\kern 8mm \text{(since $u-\frac b{\gamma^{(q-1)/e}-1}\in H_1$)}\cr
&=\left[\begin{matrix}\gamma^{(q-1)/e}&b\cr 0&1\end{matrix}\right]\overline H_1.
\end{align*}
Thus $L\subset R$.

\medskip
$2^\circ$ Now we show that $R\subset L$. Let $\overline H_1$ be the Sylow $p$-subgroup of $L$, where $H_1<\f_q$. Let $C=\{(e_1,b_1),\dots,(e_l,b_l)\}$. Since $\text{gcd}\{(q-1)/e_i:1\le i\le l\}=(q-1)/\Pi P$, we have
\[
\sum_{i=1}^l\alpha_i\frac{q-1}{e_i}=\frac{q-1}{\Pi P}
\]
for some $\alpha_1,\dots,\alpha_l\in\Bbb Z$. Then
\begin{align}\label{4.8}
L\,&\ni \left[\begin{matrix}\gamma^{(q-1)/e_1}&b_1\cr 0&1\end{matrix}\right]^{\alpha_1}\cdots \left[\begin{matrix}\gamma^{(q-1)/e_l}&b_l\cr 0&1\end{matrix}\right]^{\alpha_l}\\
&=\left[\begin{matrix}\gamma^{\sum_{i=1}^l\alpha_i(q-1)/e_i}&*\cr 0&1\end{matrix}\right]=\left[\begin{matrix}\gamma^{(q-1)/\Pi P}&*\cr 0&1\end{matrix}\right]. \nonumber
\end{align}
It follows that $H_1$ is an $\f_{p^{o_{\Pi P}(p)}}$-subspace of $\f_q$.

Clearly, $\phi^{-1}(\langle\,\bigcup_{x\in X}x\rangle)\subset H_1$. Let $a_i=\gamma^{(q-1)/e_i}$, $1\le i\le l$. Then for $1\le i,j\le l$, 
\[
L\ni\left[\begin{matrix}a_i&b_i\cr 0&1\end{matrix}\right]\left[\begin{matrix}a_j&b_j\cr 0&1\end{matrix}\right]\left[\begin{matrix}a_i&b_i\cr 0&1\end{matrix}\right]^{-1}\left[\begin{matrix}a_j&b_j\cr 0&1\end{matrix}\right]^{-1}=\left[\begin{matrix}1&(a_i-1)b_j-(a_j-1)b_i\cr 0&1\end{matrix}\right].
\]
Thus $(a_i-1)b_j-(a_j-1)b_i\in H_1$, i.e., $b_i/(a_i-1)-b_j/(a_j-1)\in H_1$. Hence $\Delta(C)\subset H_1$. Therefore
\begin{equation}\label{4.9}
\Bigl\langle\Delta(C)\cup\phi^{-1}\Bigl(\Bigl\langle\bigcup_{x\in X}x\Bigr\rangle\Bigr)\Bigr\rangle^{\f_{p^{o_{\Pi P}(p)}}}\subset H_1.
\end{equation}
Next, we take a closer look of \eqref{4.8}. In fact, we have
\begin{align*}
L\,&\ni\left[\begin{matrix}a_1&b_1\cr 0&1\end{matrix}\right]^{\alpha_1}\cdots\left[\begin{matrix}a_l&b_l\cr 0&1\end{matrix}\right]^{\alpha_l}\cr
&=\left[\begin{matrix}a_1^{\alpha_1}&(a_1^{\alpha_1}-1)b_1/(a_1-1)\cr 0&1\end{matrix}\right]\cdots\left[\begin{matrix}a_l^{\alpha_l}&(a_l^{\alpha_l}-1)b_l/(a_l-1)\cr 0&1\end{matrix}\right]\cr
&\in\left[\begin{matrix}a_1^{\alpha_1}&(a_1^{\alpha_1}-1)u\cr 0&1\end{matrix}\right]\cdots\left[\begin{matrix}a_l^{\alpha_l}&(a_l^{\alpha_l}-1)u\cr 0&1\end{matrix}\right]\overline H_1\cr
&=\left[\begin{matrix}a_1^{\alpha_1}\cdots a_l^{\alpha_l}&(a_1^{\alpha_1}\cdots a_l^{\alpha_l}-1)u_1\cr 0&1\end{matrix}\right]\overline H_1\kern 2cm \text{(by \eqref{3.8})}\cr
&=\left[\begin{matrix}\gamma^{(q-1)/\Pi P}&(\gamma^{(q-1)/\Pi P}-1)u\cr 0&1\end{matrix}\right]\overline H_1.
\end{align*}
Hence 
\begin{equation}\label{4.10}
\left[
\begin{matrix}\gamma^{(q-1)/\Pi P}&(\gamma^{(q-1)/\Pi P}-1)u\cr 0&1\end{matrix}\right]\in L.
\end{equation}
By \eqref{4.9} and \eqref{4.10}, $R\subset L$.
\end{proof}

\begin{thm}\label{T4.2}
Let $(1,0,H)\in\mathcal S$ and let $|H|=p^\beta$. Then for $0\le k\le q$,
\begin{align}\label{4.10.0}
&\mathcal N(S(1,0,H),k)=\\ 
&(1-p^{\alpha-\beta})\sum_{0\le l\le \alpha-\beta}(-1)^lp^{\binom l2}{{\alpha-\beta}\brack l}_pS_{q,k}(1,p^{\beta+l})\cr
&+p^{\alpha-\beta}\sum_{P\subset\mathcal P(|H'|-1)}\,\sum_{0\le l\le(\alpha-\beta)/o_{\Pi P}(p)}(-1)^{|P|+l}p^{o_{\Pi P}(p)\binom l2}{{(\alpha-\beta)/o_{\Pi P}(p)}\brack l}_{p^{o_{\Pi P}(p)}}\cr
&\cdot s_{q,k}(\Pi P,\, p^{\beta+lo_{\Pi P}(p)}).\nonumber
\end{align}
\end{thm}

\begin{proof}
By \eqref{4.1}, \eqref{2.2}, \eqref{4.7}, and \eqref{3.21}, we have
\begin{align}\label{4.11}
&\mathcal N(S(1,0,H),k)=\\
&\sum_{\substack{C\subset\mathcal P(|H'|-1)\times\mathcal B\cr X\subset\Bbb P_{\f_p}(\f_q/H)}}(-1)^{|C|+|X|}s_{q,k}\Bigl(\Pi P,\, \Bigl|\Bigl\langle\Delta(C)\cup\phi^{-1}\Bigl(\Bigl\langle\bigcup_{x\in X}x\Bigr\rangle\Bigr)\Bigr\rangle^{\f_{p^{o_{\Pi P}(p)}}}\Bigr|\,\Bigr). \nonumber
\end{align}
(Note. In \eqref{4.11}, $P\subset\mathcal P(|H'|-1)$ is given by \eqref{4.6} and is dependent on $C$.) In what follows, we simplify \eqref{4.11} to a more computable form. Rewrite \eqref{4.11} as
\begin{align}\label{4.12}
&\mathcal N(S(1,0,H),k)=\\
&\sum_{\substack{C\subset\mathcal P(|H'|-1)\times\mathcal B\cr U:\, \f_{p^{o_{\Pi P}(p)}}-\text{module}\cr \Delta(C)\cup H\subset U\subset\f_q}}(-1)^{|C|}s_{q,k}(\Pi P,|U|)\sum_{\substack{X\subset\Bbb P_{\f_p}(\f_q/H)\cr \langle\Delta(C)\cup\phi^{-1}(\langle\bigcup_{x\in X}x\rangle)\rangle^{\f_{p^{o_{\Pi P}(p)}}}=U}}(-1)^{|X|}. \nonumber
\end{align}
Let $C$ and $U$ in the outer sum of \eqref{4.12} be fixed. For each $\f_{p^{o_{\Pi P}(p)}}$-module $V$ with $\Delta(C)\cup H\subset V\subset U$, define
\begin{equation}\label{4.13}
f(V)=\sum_{\substack{X\subset\Bbb P_{\f_p}(\f_q/H)\cr \langle\Delta(C)\cup\phi^{-1}(\langle\bigcup_{x\in X}x\rangle)\rangle^{\f_{p^{o_{\Pi P}(p)}}}=V}}(-1)^{|X|}
\end{equation}
and 
\begin{equation}\label{4.14}
f_\le(V)=\sum_{\substack{W:\, \f_{p^{o_{\Pi P}(p)}}\text{-module}\cr \Delta(C)\cup H\subset W\subset V}}f(W).
\end{equation}
Then
\begin{equation}\label{4.15}
f_\le (V)=\sum_{X\subset \Bbb P_{\f_p}(V/H)}(-1)^{|X|}=
\begin{cases}
1&\text{if}\ V=H,\cr
0&\text{if}\ V\supsetneq H.
\end{cases}
\end{equation}
By the M\"obius version,
\begin{equation}\label{4.16}
f(V)=\sum_{\substack{W\cr \Delta(C)\cup H\subset W\subset V}}\mu_{p^{o_{\Pi P}(p)}}(W,V)f_\le (W)=
\begin{cases}
\mu_{p^{o_{\Pi P}(p)}}(H,V)&\text{if}\ \Delta(C)\subset H,\cr
0&\text{if}\ \Delta(C)\not\subset H.
\end{cases}
\end{equation}
Therefore \eqref{4.12} becomes
\begin{equation}\label{4.17}
\mathcal N(S(1,0,H),k)=\sum_{\substack{C\subset\mathcal P(|H'|-1)\times\mathcal B,\,\Delta(C)\subset H\cr U:\, \f_{p^{o_{\Pi P}(p)}}\text{-module},\, H\subset U\subset\f_q}}(-1)^{|C|}s_{q,k}(\Pi P,|U|)\,\mu_{p^{o_{\Pi P}(p)}}(H,U).
\end{equation}
Each $\emptyset\ne C\subset\mathcal P(|H'|-1)\times B$ with $\Delta(C)\subset H$ is obtained in the following way: First choose $\emptyset\ne P\subset\mathcal P(|H'|-1)$ and a coset $u+H$, where $u\in\mathcal B$. Then choose $C=\{(e,b_e):e\in P\}$, where $b_e\in\mathcal B$ is the unique element such that $b_e/(\gamma^{(q-1)/e}-1)\in u+H$. Therefore, \eqref{4.17} can be written as
\begin{align}\label{4.18}
&\mathcal N(S(1,0,H),k)\\
=\,&\sum_{U:\, \f_p\text{-module},\, H\subset U\subset\f_q}s_{q,k}(1,|U|)\,\mu_p(H,U)\cr
&+\sum_{\substack{\emptyset\ne P\subset\mathcal P(|H'|-1)\cr U:\, \f_{p^{o_{\Pi P}(p)}}\text{-module},\, H\subset U\subset\f_q}}\frac q{|H|}(-1)^{|P|}s_{q,k}(\Pi P,|U|)\,\mu_{p^{o_{\Pi P}(p)}}(H,U). \nonumber
\end{align}
Let $l=\dim_{\f_{p^{o_{\Pi P}(p)}}}(U/H)$. Then
\[
\mu_{p^{o_{\Pi P}(p)}}(H,U)=(-1)^lp^{o_{\Pi P}(p)\binom l2},
\]
and \eqref{4.18} becomes
\begin{align*}
&\mathcal N(S(1,0,H),k)=\cr
&\sum_{0\le l\le \alpha-\beta}(-1)^lp^{\binom l2}{{\alpha-\beta}\brack l}_ps_{q,k}(1,p^{\beta+l})\cr
&+p^{\alpha-\beta}\sum_{\emptyset\ne P\subset\mathcal P(|H'|-1)}\,\sum_{0\le l\le(\alpha-\beta)/o_{\Pi P}(p)}(-1)^{|P|+l}p^{o_{\Pi P}(p)\binom l2}{{(\alpha-\beta)/o_{\Pi P}(p)}\brack l}_{p^{o_{\Pi P}(p)}}\cr
&\cdot s_{q,k}(\Pi P,\, p^{\beta+lo_{\Pi P}(p)}). 
\end{align*}
\end{proof} 

\subsection{$S=S(\gamma^{(q-1)/d},0,H)$, $1<d\mid q-1$}\

Let $(\gamma^{(q-1)/d},0,H)\in\mathcal S$, where $1<d\mid q-1$. The computation of $\mathcal N(S(\gamma^{(q-1)/d},0,$ $H),k)$ is similar to Subsection~\ref{s4.1}. By Proposition~\ref{P3.3} (ii), the immediate supergroups of $S(\gamma^{(q-1)/d},0,H)$ in $\text{AGL}(1,\f_q)$ are precisely
\begin{equation}\label{4.19}
S(\gamma^{(q-1)/de},0,H),\qquad e\in\mathcal P\Bigl(\frac{|H'|-1}d\Bigr),
\end{equation}
and 
\begin{equation}\label{4.20}
S(\gamma^{(q-1)/d},0,\phi^{-1}(x)),\qquad x\in\Bbb P_{\f_{p^{o_d(p)}}}(\f_q/H).
\end{equation}

\begin{lem}\label{L4.2.0}
Let $1<d\mid q-1$, $P\subset\mathcal P((|H'|-1)/d)$ and $X\subset \Bbb P_{\f_{p^{o_d(p)}}}(\f_q/H)$. Then 
\begin{align}\label{4.21}
&\Bigl\langle S(\gamma^{(q-1)/d},0,H)\cup\Bigl(\bigcup_{e\in P}S(\gamma^{(q-1)/de},0,H)\Bigr)\cup\Bigl(\bigcup_{x\in X}S(\gamma^{(q-1)/d},0,\phi^{-1}(x))\Bigr)\Bigr\rangle\\
=\,&S\Bigl(\gamma^{(q-1)/d\Pi P},\, 0,\, \Bigl\langle\phi^{-1}\Bigl(\Bigl\langle\bigcup_{x\in X}x\Bigr\rangle\Bigr)\Bigr\rangle^{\f_{p^{o_{d\Pi P}(p)}}}\Bigr). \nonumber
\end{align}
\end{lem}

\begin{proof}
The proof of Lemma~\ref{L4.2.0} is similar to but simpler than that of Lemma~\ref{L4.1}. We leave the details for the reader.
\end{proof}

\begin{thm}\label{T4.3}
Let $(\gamma^{(q-1)/d},0,H)\in\mathcal S$, where $1<d\mid q-1$, and let $|H|=p^\beta$. Then for $0\le k\le q$, 
\begin{align}\label{4.21.0}
&\mathcal N(S(\gamma^{(q-1)/d},0,H),k)=\\
&\sum_{P\subset\mathcal P((|H'|-1)/d)}\,\sum_{0\le l\le(\alpha-\beta)/o_{d\Pi P}(p)}(-1)^{|P|+l}p^{o_{d\Pi P}(p)\binom l2}{{(\alpha-\beta)/o_{d\Pi P}(p)}\brack l}_{p^{o_{d\Pi P}(p)}}\cr 
\noalign{\vspace{1mm}}
&\cdot s_{q,k}(d\,\Pi P,\, p^{\beta+lo_{d\Pi P}(p)}). \nonumber
\end{align}
\end{thm}

\begin{proof}
By \eqref{4.1}, \eqref{2.2}, \eqref{4.21}, and \eqref{3.21},
\begin{align}\label{4.22}
&\mathcal N(S(\gamma^{(q-1)/d},0,H),k)\\
=\,&\sum_{\substack{P\subset \mathcal P((|H'|-1)/d)\cr X\subset{\Bbb P\,_{\f}}\!_{p^{o_d(p)}}(\f_q/H)}}(-1)^{|P|+|X|}s_{q,k}\Bigl(d\,\Pi P,\,\Bigl|\Bigl\langle \phi^{-1}\Bigl(\Bigl\langle\bigcup_{x\in X}x\Bigr\rangle\Bigr)\Bigr\rangle^{\f_{p^{o_{d\Pi P}(p)}}}\Bigr|\Bigr)\cr
=\,&\sum_{\substack{P\subset \mathcal P((|H'|-1)/d)\cr U:\,\f_{p^{o_{d\Pi P}(p)}}\text{-module}\cr H\subset U\subset\f_q}}(-1)^{|P|}s_{q,k}(d\,\Pi P, |U|)\sum_{\substack{X\subset{\Bbb P\,_{\f}}\!_{p^{o_d(p)}}(\f_q/H)\cr \langle\phi^{-1}(\langle\bigcup_{x\in X}x\rangle)\rangle^{\f_{p^{o_{d\Pi P}(p)}}}=U}}(-1)^{|X|}. \nonumber
\end{align}
By the same argument that led to \eqref{4.16}, we find that the inner sum in \eqref{4.22} equals $\mu_{p^{o_{d\Pi P}(p)}(p)}(H,U)$. Hence 
\begin{align}\label{4.23}
&\mathcal N(S(\gamma^{(q-1)/d},0,H),k)\\
=\,&\sum_{\substack{P\subset \mathcal P((|H'|-1)/d)\cr U:\,\f_{p^{o_{d\Pi P}(p)}}\text{-module},\, H\subset U\subset\f_q}}(-1)^{|P|}s_{q,k}(d\Pi P, |U|)\,\mu_{p^{o_{d\Pi P}(p)}}(H,U). \nonumber
\end{align}
Let $l=\dim_{\f_{p^{o_{d\Pi P}(p)}}}(U/H)$. Then
\begin{align*}
&\mathcal N(S(\gamma^{(q-1)/d},0,H),k)=\cr
&\sum_{P\subset \mathcal P((|H'|-1)/d)}\,\sum_{0\le l\le(\alpha-\beta)/o_{d\Pi P}(p)}(-1)^{|P|+l}p^{o_{d\Pi P}(p)\binom l2}{{(\alpha-\beta)/o_{d\Pi P}(p)}\brack l}_{p^{o_{d\Pi P}(p)}}\cr
\noalign{\vspace{1mm}}
&\cdot s_{q,k}(d\,\Pi P,\, p^{\beta+lo_{d\Pi P}(p)}).
\end{align*}
\end{proof}

\begin{rmk}\label{R4.4}\rm
\begin{itemize}
\item[] \hspace{-10mm} (i)\; It follows from \eqref{4.3.0}, \eqref{3.21} and Remark~\ref{R3.5} that $\mathcal N(S(\gamma^{(q-1)/d},$ $0,H),k)=0$ unless $k\equiv 0$ or $p^\beta \pmod{dp^\beta}$, where $d\mid p-1$ and $|H|=p^\beta$.

\medskip

\item[(ii)]
Assume that $k=0$ or $q$. It is clear that 
\begin{equation}\label{4.25}
\mathcal N(S,0)=\mathcal N(S,q)=
\begin{cases}
1&\text{if}\ S=S(\gamma,0,\f_q),\cr
0&\text{otherwise}.
\end{cases}
\end{equation}

\medskip

\item[(iii)] Assume that $k=1$. Let $x\in\f_q$. For $a\in\f_q^*$ and $b\in\f_q$, $ax+b=x$ if and only if $b=(1-a)x$. Hence
\[
\text{AGL}(1,\f_q)_{\{x\}}=\Bigl\{\left[\begin{matrix} a&(1-a)x\cr 0&1\end{matrix}\right]:a\in\f_q^*\Bigr\}=S(\gamma,(1-\gamma)x,0).
\]
Therefore for $S<\text{AGL}(1,\f_q)$,
\begin{equation}\label{4.24}
\mathcal N(S,1)=
\begin{cases}
2&\text{if}\ q=2,\ S=S(1,0,0),\cr
1&\text{if}\ q>2,\ S=S(\gamma,b,0),\ b\in\f_q,\cr
0&\text{otherwise}.
\end{cases}
\end{equation}

\medskip

\item[(iv)] Assume that $k=2$. Let $\{x,y\}\in\binom{\f_q}2$. For $a\in\f_q^*$ and $b\in\f_q$,
\[
\begin{cases}
ax+b=y,\cr
ay+b=x
\end{cases}
\]
if and only if $a=-1$ and $b=x+y$. Hence
\[
\text{AGL}(1,\f_q)_{\{x,y\}}=\Bigl\{\left[\begin{matrix} 1&0\cr 0&1\end{matrix}\right],\ \left[\begin{matrix} -1&x+y\cr 0&1\end{matrix}\right]\Bigr\}=
\begin{cases}
S(1,0,(x+y)\f_2)&\text{if $q$ is even},\cr
S(-1,x+y,0)&\text{if $q$ is odd}.
\end{cases}
\]
Hence 
\begin{equation}\label{4.26}
\mathcal N(S,2)=
\begin{cases}
q/2&\text{if $q$ is even},\ S=S(1,0,H),\ |H|=2, \cr
(q-1)/2&\text{if $q$ is odd},\ S=S(\gamma^{(q-1)/2},b,0),\ b\in\f_q,\cr
0&\text{otherwise}.
\end{cases}
\end{equation}

\medskip

\item[(v)] Statements (i) -- (iv) agree with \eqref{4.10.0} and \eqref{4.21.0} as they should. Obviously, \eqref{4.10.0} and \eqref{4.21.0} imply (i). However, to derive \eqref{4.25} -- \eqref{4.26} from \eqref{4.10.0} and \eqref{4.21.0}, it requires some computation; we leave the details to the reader.
\end{itemize}
\end{rmk}


\section{Computation of $\mathcal N(S(\gamma^{(q-1)/d},0,H),k)$}\label{s5}

The formulas \eqref{4.10.0} and \eqref{4.21.0} for $\mathcal N(S(\gamma^{(q-1)/d},0,H),k)$ depend on $q$, $k$, $d$, $|H|$ and $|H'|$ but not on $H$. Thus, with $q$, $k$, $d$ given, we only have to know the possible pairs of values $(|H|,|H'|)$.

Recall that $q=p^\alpha$ and $\f_q(\gamma^{(q-1)/d})=\f_{p^{o_d(p)}}$.
 
\begin{lem}\label{L5.1}
The possible values of $(|H|,|H'|)$, where $H\subset\f_q$ is an $\f_{p^{o_d(p)}}$-subspace, are precisely
\[
(1,p^\alpha),\ (p^\alpha,p^\alpha),\ \text{and}\ (p^{o_d(p)ij},p^{o_d(p)i}),\ i\mid(\alpha/o_d(p)),\ i<\alpha/o_d(p),\ 0<j<\alpha/o_d(p)i.
\]
\end{lem}

\begin{proof}
This follows from Lemma~\ref{L5.2}.
\end{proof}

\begin{lem}\label{L5.2}
Let $\f_r\subset\f_q$ and $0\le d\le[\f_q:\f_r]$. Then there exists $H<\f_q$ such that $H'=\f_r$ and $\dim_{\f_r}H=d$ if and only if 
\begin{itemize}
\item[(i)] $[\f_q:\f_r]=1$, or
\item[(ii)] $[\f_q:\f_r]>1$ and $0<d<[\f_q:\f_r]$.
\end{itemize}
\end{lem}

\begin{proof}
The proof is essentially the same as that of \cite[Theorem~4.14]{Sun-TJM-2010}.

\medskip

($\Rightarrow$) Obvious.

\medskip
($\Leftarrow$) Let $n=[\f_q:\f_r]$. If $n=1$, let $H=0$ for $d=0$ and $H=\f_q$ for $d=1$. Assume that $n>1$ and $0<d<n$. Let $e_1,\dots,e_l$ be the distinct prime factors of $n$. Clearly, $l\le n-1$. Choose $a_i\in\f_{r^{e_i}}\setminus\f_r$, $1\le i\le l$. Clearly, $a_1,\dots,a_l, a_{l+1}(=1)$ are linearly independent over $\f_r$. Extend them to an $\f_r$-basis $a_i$, $1\le i\le n$, of $\f_q$. Let $L=\{\sum_{i=1}^nx_ia_i:x_i\in\f_r,\ x_1+\cdots+ x_l=0\}\subset\f_q$. Then $\dim_{\f_r}L=n-1$, $\f_r\subset L$, and $a_i\notin L$ for all $1\le i\le l$. Let $H$ be an $\f_r$-subspace of $L$ such that $\f_r\subset H$ and $\dim_{\f_r}H=d$. For each $1\le i\le l$, since $1\in H$ and $a_i\notin H$, we have $\f_{r^{e_i}}\not\subset H'$. Hence $H'=\f_r$.
\end{proof}

The input data for \eqref{4.10.0} and \eqref{4.21.0} are $p,\alpha,k,d,i,j$, where $0\le k\le p^\alpha/2$, $d\mid p^\alpha-1$, $i\mid(\alpha/o_d(p))$, $j=0,1$ if $i=\alpha/o_d(p)$ and $0<j<\alpha/o_d(p)i$ if $i<\alpha/o_d(p)$, $k\equiv 0$ or $p^{o_d(p)ij}\pmod{dp^{o_d(p)ij}}$. These data produce $q=p^\alpha$, $|H|=p^\beta$, where $\beta=o_d(p)ij$, and $|H'|=p^{o_d(p)i}$. Afterwards, $\mathcal N(S(\gamma^{(q-1)/d},0,H),k)$ can be computed (with computer assistance) by \eqref{4.10.0} (for $d=1$) and by \eqref{4.21.0} (for $d>1$). A Mathematica code is included in Appendix~A1; the computational results for $q\le 101$ are given in Table~\ref{Tb1} in Appendix~A2.

\section{Remarks and Open Questions}

Let $\Bbb S_q$ denote the set of all subgroups of $\text{AGL}(1,\f_q)$. Fix $0\le k\le q$. There are two maps between $\Bbb S_q$ and $\binom{\f_q} k$:
\begin{equation}\label{6.1}
\begin{array}{cccl}
(\ )':& \Bbb S_q& \longrightarrow & \displaystyle \binom{\f_q} k \vspace{2mm}\cr
& S&\longmapsto & S'=\Bigl\{B\in \displaystyle \binom{\f_q}k: \sigma(B)=B\ \text{for all}\ \sigma\in S\Bigr\},
\end{array}
\end{equation}
and
\begin{equation}\label{6.2}
\begin{array}{cccl}
(\ )':& \displaystyle \binom{\f_q} k& \longrightarrow & \Bbb S_q \vspace{2mm}\cr
& B&\longmapsto & B'=\text{AGL}(1,\f_q)_B.
\end{array}
\end{equation}
These two maps enjoy the properties of a ``Galois correspondence'': both $(\ )'$ are inclusion reversing; $S\subset S''$ and $B\subset B''$ for all $S\in\Bbb S_q$ and $B\in\binom{\f_q}k$; $(\ )'''=(\ )'$. The map $(\ )':\Bbb S_q\to \binom{\f_q}k$ has been described in Subsection~\ref{s3.3}. However, the map $(\ )':\binom{\f_q}k\to \Bbb S_q$ is not well understood; in particular, Question~\ref{Q1.2} asks about the image of this map

Theorem~\ref{T4.2} and Theorem~\ref{T4.3} show that one can ``compute'' an answer to Question~\ref{Q1.2} (for $\text{AGL}(1,\f_q)$) by computing the function $\mathcal N=\mathcal N(S(\gamma^{(q-1)/d},0,H),k)$. We do not know if it is possible to find a more theoretic and general solution to Question~\ref{Q1.2} for $\text{AGL}(1,\f_q)$.

Sun \cite{Sun-TJM-2010} gave several constructions of $B\in\binom{\f_q}k$ with specified stabilizers. Those results can be stated in terms of nonvanishing of the function $\mathcal N$ as follows:

\begin{thm}[{\cite[Theorem~4.14]{Sun-TJM-2010}}, Lemma~\ref{L5.1}]\label{T6.1}
Let $H<\f_q$ with $H'=\f_{p^i}$. Then
\[
\text{\rm AGL}(1,\f_q)_H=S(\gamma^{(q-1)/(p^i-1)},0,H).
\]
In particular, for $(\beta,i)=(0,\alpha)$ or $(\alpha,\alpha)$, or $0<\beta<\alpha$ and $i\mid \text{\rm gcd}(\alpha,\beta)$, there exists $H<\f_q$ with $|H|=p^\beta$ and $|H'|=p^i$ such that 
\[
\mathcal N(S(\gamma^{(q-1)/(p^i-1)}, 0,H),\, p^\beta)>0.
\]
\end{thm}

\begin{thm}[{\cite[Theorem~3.5]{Sun-TJM-2010}}]\label{T6.2}
Assume that $3\le k\le q/2$, $p\nmid k$, $d\mid q-1$, $k\equiv 0$ or $1\pmod d$, and $(q,k,d)\ne(7,3,1)$. Then
\[
\mathcal N(S(\gamma^{(q-1)/d}, 0,\{0\}),\, k)>0.
\]
\end{thm}

\begin{thm}[{\cite[Theorem~4.4]{Sun-TJM-2010}}]\label{T6.3}
Assume that $3\le k\le q/2$ and
\[
k\equiv
\begin{cases}
2\pmod 4&\text{if}\ p=2,\cr
0\pmod p&\text{if}\ p>2.
\end{cases}
\]
Then
\[
\mathcal N(S(1, 0,\{0\}),\, k)>0.
\]
(This corresponds to $d=1$ and $\beta=0$ in \eqref{para-cond}.)
\end{thm}

\begin{thm}[{\cite[Theorem~4.17]{Sun-TJM-2010}}]\label{T6.4}
Let $\beta, d, m$ be positive integers such that $\beta\mid\alpha$, $\beta<\alpha$, $p^\beta>2$ and $d+m\le \alpha/\beta$. Let $c_i$, $k_i$ ($1\le i\le m$) be positive integers such that $c_i\mid p^\beta-1$, $\text{\rm gcd}(c_1,\dots,c_m)=c>1$, and $p\nmid k_i$. Further assume that for each $1\le i\le m$, $c_i\mid k_i$ and $2\le k_i\le p^\beta-3$, or $c_i\mid k_i-1$ and $3\le k_i\le p^\beta-2$. Then
\[
\mathcal N(S(\gamma^{(q-1)/c}, 0,H),\,k_1\cdots k_mp^{\beta d})>0
\]
for some $H<\f_q$ with $|H|=p^{\beta d}$ and $|H'|=p^\beta$.
\end{thm}

Our numerical results are consistent with the above theorems. Moreover,
the numerical results in Appendix~A2 provide plenty material for conjectures and open questions about the function $\mathcal N=\mathcal N(S(\gamma^{(q-1)/d},0,H),k)$. We include a few open questions to invite further work. We always assume that $k\equiv 0$ or $|H|\pmod{d|H|}$, which is necessary for $\mathcal N\ne 0$ (Remark~\ref{R4.4} (i)).

\begin{ques}\label{Q6.1}
For $3\le k\le q/2$, $\mathcal N=0$ only if $p\mid k$ (Theorem~\ref{T6.2}). Determine exactly when $\mathcal N=0$.
\end{ques}

\begin{ques}[A subquestion of Question~\ref{Q6.1}]\label{Q6.2}
For $p=2$ and $3\le k\le q/2$, it appears that $\mathcal N=0$ only if $k$ is a power of $2$. Is this true?
\end{ques}

\begin{ques}\label{Q6.3}
In certain cases where $\mathcal N(S,k)$ is positive but small, determine those $B\in\binom{\f_q}k$ such that $\text{\rm AGL}(1,\f_q)_B=S$.
\end{ques}

\begin{ques}\label{Q6.4} 
Determine the $p$-adic valuation of $\mathcal N(S,k)$ in \eqref{4.10.0} and \eqref{4.21.0}. This could give $\mathcal N(S,k)\ne 0$ without computations involving large integers.
\end{ques}

The work of the present paper relies on a fairly detailed description of the lattice of subgroups of $\text{AGL}(1,\f_q)$. In fact, based on the approach of the present paper, we are able to determine the M\"obius function of the lattice of subgroups of $\text{\rm AGL}(1,$ $\f_q)$ \cite{Hou-ppt}.


\section*{Appendix}

\subsection*{A1. A Mathematica code for computing $\mathcal N(S(\gamma^{(q-1)/d},0,H),k)$}\

In the following Mathematica program, the input is $q=p^\alpha$; the output is the array $(k,d,o_d(p),i,j,\beta,\mathcal N)$, where
\begin{align}\label{para-cond}\tag{A1.1}
&0\le k\le p^\alpha/2,\\
&d\mid p^\alpha-1,\cr
&i\mid(\alpha/o_d(p)),\cr
&\begin{cases}
j=0,1&\text{if}\ i=\alpha/o_d(p),\cr
0<j<\alpha/o_d(p)i&\text{if}\ i<\alpha/o_d(p),
\end{cases}\cr
&\beta=o_d(\alpha)ij,\cr
&k\equiv 0\ \text{or}\ p^\beta\pmod{dp^\beta},\cr
&\mathcal N=\mathcal N(S(\gamma^{(q-1)/d},0,H),k),\ \text{where}\ |H|=p^\beta,\; |H'|=p^{o_d(p)i}. \nonumber
\end{align}

\medskip

\begin{mmaCell}{Code}

MyBinom[x_, y_] := If[IntegerQ[y], Binomial[x, y], 0];
s[q_, k_, x_, y_] := 
  MyBinom[(q - y)/(x*y), k/(x*y)] + 
   MyBinom[(q - y)/(x*y), (k - y)/(x*y)];

p = 2;
alpha = 6;
q = p^alpha;
Print["q= ", p, "^", alpha];
divd = Sort[Divisors[q - 1]];
For[k = 0, k <= q/2, k++,
  For[u = 1, u <= Length[divd], u++,
   d = Extract[divd, {u}];
   
   odp = MultiplicativeOrder[p, d];
   divi = Sort[Divisors[alpha/odp]];
   For[v = 1, v <= Length[divi], v++,
    i = Extract[divi, {v}];
    If[i == alpha/odp, j0 = 0; j1 = 1, j0 = 1; j1 = alpha/(odp*i) - 1];
    
    For [j = j0, j <= j1, j++,
     beta = odp*i*j;
     
     condition = Mod[k, d*p^beta];
     If[condition != 0 && condition != p^beta, Continue[]];
     
     x = (p^(odp*i) - 1)/d;
     If[x == 1, PSet = {},
      PSet = FactorInteger[x][[All, 1]]
      ];
     PP = Subsets[PSet];
     
     If[d == 1, 
      Num = (1 - p^(alpha - beta))*
        Sum[(-1)^l*p^Binomial[l, 2]*QBinomial[alpha - beta, l, p]*
          s[q, k, 1, p^(beta + l)], {l, 0, alpha - beta}],
      Num = 0];
     
     N1 = 0;
     For[t = 1, t <= Length[PP], t++,
      P = Extract[PP, {t}];
      piP = Times @@ P;
      odpip = MultiplicativeOrder[p, d*piP];
      N1 = 
       N1 + Sum[(-1)^(Length[P] + l)*p^(odpip*Binomial[l, 2])*
          QBinomial[(alpha - beta)/odpip, l, p^odpip]*
          s[q, k, d*piP, p^(beta + l*odpip)], {l, 
          0, (alpha - beta)/odpip}];
      ];
     If[d == 1,
      Num = Num + p^(alpha - beta)*N1,
      Num = Num + N1];
     
     Print["k= ", k, ", d= ", d, ", o_d(p)= ", odp, ", i= ", i, 
      ", j= ", j, ", beta= " , beta, ",  N= ", Num];
     ];
    ];
   ];
  ];

\end{mmaCell}


\subsection*{A2. Numerical results}\

Table~\ref{Tb1} gives the values of $\mathcal N(S(\gamma^{(q-1)/d},0,H),k)$ for $q\le 101$. To recall the meanings of and the conditions on the parameters, refer to \eqref{para-cond} in Appendix~A1.



\newpage



\input table1

\end{document}

%% file: table1.tex



\begin{table}
\caption{Values of $\mathcal N=\mathcal N(S(\gamma^{(q-1)/d},0,H),k)$, $q\le 101$}\label{Tb1}
\[
\begin{tabular}{ll|llllll|l}
$p$ & $\alpha$ & $k$ & $d$ & $o_d(p)$ & $i$ & $j$  & $\beta$ & $\mathcal N$\\ \hline


2&1
&0&1&1&1&0&0&0\\

&&0&1&1&1&1&1&1\\

&&1&1&1&1&0&0&2\\ \hline


2&2
&0&1&1&1&1&1&0\\

&&0&1&1&2&0&0&0\\

&&0&1&1&2&1&2&0\\

&&0&3&2&1&0&0&0\\

&&0&3&2&1&1&2&1\\

&&1&1&1&2&0&0&0\\

&&1&3&2&1&0&0&1\\

&&2&1&1&1&1&1&2\\

&&2&1&1&2&0&0&0\\ \hline


2&3
&0&1&1&1&1&1&0\\

&&0&1&1&1&2&2&0\\

&&0&1&1&3&0&0&0\\

&&0&1&1&3&1&3&0\\

&&0&7&3&1&0&0&0\\

&&0&7&3&1&1&3&1\\

&&1&1&1&3&0&0&0\\

&&1&7&3&1&0&0&1\\

&&2&1&1&1&1&1&4\\

&&2&1&1&3&0&0&0\\

&&3&1&1&3&0&0&56\\

&&4&1&1&1&1&1&0\\

&&4&1&1&1&2&2&2\\

&&4&1&1&3&0&0&56\\ \hline


2&4
&0&1&1&1&1&1&0\\

&&0&1&1&1&2&2&0\\

&&0&1&1&1&3&3&0\\

&&0&1&1&2&1&2&0\\

&&0&1&1&4&0&0&0\\

&&0&1&1&4&1&4&0\\

&&0&3&2&1&1&2&0\\

&&0&3&2&2&0&0&0\\

&&0&3&2&2&1&4&0\\

&&0&5&4&1&0&0&0\\

&&0&5&4&1&1&4&0\\

&&0&15&4&1&0&0&0\\

&&0&15&4&1&1&4&1\\

&&1&1&1&4&0&0&0\\

&&1&3&2&2&0&0&0\\

&&1&5&4&1&0&0&0\\

&&1&15&4&1&0&0&1\\

\hline
\end{tabular}
\]
\end{table}

\setcounter{table}{0}
\begin{table}
\caption{Continued}
\[
\begin{tabular}{ll|llllll|l}
$p$ & $\alpha$ & $k$ & $d$ & $o_d(p)$ & $i$ & $j$  & $\beta$ & $\mathcal N$\\ \hline

2&4
&2&1&1&1&1&1&8\\

&&2&1&1&4&0&0&0\\

&&3&1&1&4&0&0&480\\

&&3&3&2&2&0&0&5\\

&&4&1&1&1&1&1&0\\

&&4&1&1&1&2&2&4\\

&&4&1&1&2&1&2&0\\

&&4&1&1&4&0&0&1680\\

&&4&3&2&1&1&2&1\\

&&4&3&2&2&0&0&0\\

&&5&1&1&4&0&0&4320\\

&&5&5&4&1&0&0&3\\

&&6&1&1&1&1&1&56\\

&&6&1&1&4&0&0&6960\\

&&6&3&2&2&0&0&10\\

&&6&5&4&1&0&0&3\\

&&7&1&1&4&0&0&11280\\

&&7&3&2&2&0&0&10\\

&&8&1&1&1&1&1&56\\

&&8&1&1&1&2&2&0\\

&&8&1&1&1&3&3&2\\

&&8&1&1&2&1&2&0\\

&&8&1&1&4&0&0&12000\\ \hline


2&5
&0&1&1&1&1&1&0\\

&&0&1&1&1&2&2&0\\

&&0&1&1&1&3&3&0\\

&&0&1&1&1&4&4&0\\

&&0&1&1&5&0&0&0\\

&&0&1&1&5&1&5&0\\

&&0&31&5&1&0&0&0\\

&&0&31&5&1&1&5&1\\

&&1&1&1&5&0&0&0\\

&&1&31&5&1&0&0&1\\

&&2&1&1&1&1&1&16\\

&&2&1&1&5&0&0&0\\

&&3&1&1&5&0&0&4960\\

&&4&1&1&1&1&1&0\\

&&4&1&1&1&2&2&8\\

&&4&1&1&5&0&0&34720\\

&&5&1&1&5&0&0&201376\\

&&6&1&1&1&1&1&560\\

&&6&1&1&5&0&0&888832\\

&&7&1&1&5&0&0&3365856\\

\hline
\end{tabular}
\]
\end{table}

\setcounter{table}{0}
\begin{table}
\caption{Continued}
\[
\begin{tabular}{ll|llllll|l}
$p$ & $\alpha$ & $k$ & $d$ & $o_d(p)$ & $i$ & $j$  & $\beta$ & $\mathcal N$\\ \hline

2&5
&8&1&1&1&1&1&1680\\

&&8&1&1&1&2&2&0\\

&&8&1&1&1&3&3&4\\

&&8&1&1&5&0&0&10465600\\

&&9&1&1&5&0&0&28048800\\

&&10&1&1&1&1&1&4368\\

&&10&1&1&5&0&0&64376832\\

&&11&1&1&5&0&0&129024480\\

&&12&1&1&1&1&1&7168\\

&&12&1&1&1&2&2&56\\

&&12&1&1&5&0&0&225561952\\

&&13&1&1&5&0&0&347373600\\

&&14&1&1&1&1&1&11440\\

&&14&1&1&5&0&0&471080960\\

&&15&1&1&5&0&0&565722720\\

&&16&1&1&1&1&1&12000\\

&&16&1&1&1&2&2&56\\

&&16&1&1&1&3&3&0\\

&&16&1&1&1&4&4&2\\

&&16&1&1&5&0&0&600699648\\ \hline


2&6
&0&1&1&1&1&1&0\\

&&0&1&1&1&2&2&0\\

&&0&1&1&1&3&3&0\\

&&0&1&1&1&4&4&0\\

&&0&1&1&1&5&5&0\\

&&0&1&1&2&1&2&0\\

&&0&1&1&2&2&4&0\\

&&0&1&1&3&1&3&0\\

&&0&1&1&6&0&0&0\\

&&0&1&1&6&1&6&0\\

&&0&3&2&1&1&2&0\\

&&0&3&2&1&2&4&0\\

&&0&3&2&3&0&0&0\\

&&0&3&2&3&1&6&0\\

&&0&7&3&1&1&3&0\\

&&0&7&3&2&0&0&0\\

&&0&7&3&2&1&6&0\\

&&0&9&6&1&0&0&0\\

&&0&9&6&1&1&6&0\\

&&0&21&6&1&0&0&0\\

&&0&21&6&1&1&6&0\\

&&0&63&6&1&0&0&0\\

&&0&63&6&1&1&6&1\\

\hline
\end{tabular}
\]
\end{table}

\setcounter{table}{0}
\begin{table}
\caption{Continued}
\[
\begin{tabular}{ll|llllll|l}
$p$ & $\alpha$ & $k$ & $d$ & $o_d(p)$ & $i$ & $j$  & $\beta$ & $\mathcal N$\\ \hline

2&6
&1&1&1&6&0&0&0\\

&&1&3&2&3&0&0&0\\

&&1&7&3&2&0&0&0\\

&&1&9&6&1&0&0&0\\

&&1&21&6&1&0&0&0\\

&&1&63&6&1&0&0&1\\

&&2&1&1&1&1&1&32\\

&&2&1&1&6&0&0&0\\

&&3&1&1&6&0&0&40320\\

&&3&3&2&3&0&0&21\\

&&4&1&1&1&1&1&0\\

&&4&1&1&1&2&2&16\\

&&4&1&1&2&1&2&0\\

&&4&1&1&6&0&0&624960\\

&&4&3&2&1&1&2&1\\

&&4&3&2&3&0&0&0\\

&&5&1&1&6&0&0&7624512\\

&&6&1&1&1&1&1&4960\\

&&6&1&1&6&0&0&74648448\\

&&6&3&2&3&0&0&210\\

&&7&1&1&6&0&0&621202176\\

&&7&3&2&3&0&0&210\\

&&7&7&3&2&0&0&9\\

&&8&1&1&1&1&1&34720\\

&&8&1&1&1&2&2&0\\

&&8&1&1&1&3&3&8\\

&&8&1&1&2&1&2&0\\

&&8&1&1&3&1&3&0\\

&&8&1&1&6&0&0&4423966848\\

&&8&7&3&1&1&3&1\\

&&8&7&3&2&0&0&0\\

&&9&1&1&6&0&0&27540499392\\

&&9&3&2&3&0&0&1323\\

&&9&9&6&1&0&0&7\\

&&10&1&1&1&1&1&201376\\

&&10&1&1&6&0&0&151460443008\\

&&10&3&2&3&0&0&1323\\

&&10&9&6&1&0&0&7\\

&&11&1&1&6&0&0&743595781824\\

&&12&1&1&1&1&1&888832\\

&&12&1&1&1&2&2&560\\

&&12&1&1&2&1&2&480\\

&&12&1&1&6&0&0&3284157965760\\

\hline
\end{tabular}
\]
\end{table}

\setcounter{table}{0}
\begin{table}
\caption{Continued}
\[
\begin{tabular}{ll|llllll|l}
$p$ & $\alpha$ & $k$ & $d$ & $o_d(p)$ & $i$ & $j$  & $\beta$ & $\mathcal N$\\ \hline

2&6
&12&3&2&1&1&2&5\\

&&12&3&2&3&0&0&5880\\

&&13&1&1&6&0&0&13136858429184\\

&&13&3&2&3&0&0&5985\\

&&14&1&1&1&1&1&3365856\\

&&14&1&1&6&0&0&47855487907584\\

&&14&7&3&2&0&0&36\\

&&15&1&1&6&0&0&159518998558080\\

&&15&3&2&3&0&0&20349\\

&&15&7&3&2&0&0&36\\

&&16&1&1&1&1&1&10465600\\

&&16&1&1&1&2&2&1680\\

&&16&1&1&1&3&3&0\\

&&16&1&1&1&4&4&4\\

&&16&1&1&2&1&2&1680\\

&&16&1&1&2&2&4&0\\

&&16&1&1&3&1&3&0\\

&&16&1&1&6&0&0&488526275349504\\

&&16&3&2&1&1&2&0\\

&&16&3&2&1&2&4&1\\

&&16&3&2&3&0&0&20328\\

&&17&1&1&6&0&0&1379370175283520\\

&&18&1&1&1&1&1&28048800\\

&&18&1&1&6&0&0&3601687020470784\\

&&18&3&2&3&0&0&54243\\

&&18&9&6&1&0&0&21\\

&&19&1&1&6&0&0&8719878122149824\\

&&19&3&2&3&0&0&54243\\

&&19&9&6&1&0&0&21\\

&&20&1&1&1&1&1&64376832\\

&&20&1&1&1&2&2&4368\\

&&20&1&1&2&1&2&4368\\

&&20&1&1&6&0&0&19619721724067136\\

&&21&1&1&6&0&0&41107996870488576\\

&&21&3&2&3&0&0&116277\\

&&21&7&3&2&0&0&81\\

&&21&21&6&1&0&0&3\\

&&22&1&1&1&1&1&129024480\\

&&22&1&1&6&0&0&80347440307248576\\

&&22&3&2&3&0&0&116277\\

&&22&7&3&2&0&0&81\\

&&22&21&6&1&0&0&3\\

&&23&1&1&6&0&0&146721427591999680\\

&&24&1&1&1&1&1&225561952\\

\hline
\end{tabular}
\]
\end{table}

\setcounter{table}{0}
\begin{table}
\caption{Continued}
\[
\begin{tabular}{ll|llllll|l}
$p$ & $\alpha$ & $k$ & $d$ & $o_d(p)$ & $i$ & $j$  & $\beta$ & $\mathcal N$\\ \hline

2&6
&24&1&1&1&2&2&7168\\

&&24&1&1&1&3&3&56\\

&&24&1&1&2&1&2&7008\\

&&24&1&1&3&1&3&56\\

&&24&1&1&6&0&0&250649091241508736\\

&&24&3&2&1&1&2&10\\

&&24&3&2&3&0&0&203280\\

&&25&1&1&6&0&0&401038568738442432\\

&&25&3&2&3&0&0&203490\\

&&26&1&1&1&1&1&347373600\\

&&26&1&1&6&0&0&601557831242661888\\

&&27&1&1&6&0&0&846636978456505152\\

&&27&3&2&3&0&0&293895\\

&&27&9&6&1&0&0&35\\

&&28&1&1&1&1&1&471080960\\

&&28&1&1&1&2&2&11440\\

&&28&1&1&2&1&2&11280\\

&&28&1&1&6&0&0&1118770263280885824\\

&&28&3&2&1&1&2&10\\

&&28&3&2&3&0&0&293685\\

&&28&7&3&2&0&0&126\\

&&28&9&6&1&0&0&35\\

&&29&1&1&6&0&0&1388818294740289728\\

&&29&7&3&2&0&0&126\\

&&30&1&1&1&1&1&565722720\\

&&30&1&1&6&0&0&1620287974867242240\\

&&30&3&2&3&0&0&352716\\

&&31&1&1&6&0&0&1777090076042968512\\

&&31&3&2&3&0&0&352716\\

&&32&1&1&1&1&1&600699648\\

&&32&1&1&1&2&2&12000\\

&&32&1&1&1&3&3&56\\

&&32&1&1&1&4&4&0\\

&&32&1&1&1&5&5&2\\

&&32&1&1&2&1&2&12000\\

&&32&1&1&2&2&4&0\\

&&32&1&1&3&1&3&56\\

&&32&1&1&6&0&0&1832624103090622464\\ \hline


3&1
&0&1&1&1&0&0&0\\

&&0&1&1&1&1&1&0\\

&&0&2&1&1&0&0&0\\

&&0&2&1&1&1&1&1\\

&&1&1&1&1&0&0&0\\

&&1&2&1&1&0&0&1\\ \hline

\end{tabular}
\]
\end{table}

\setcounter{table}{0}
\begin{table}
\caption{Continued}
\[
\begin{tabular}{ll|llllll|l}
$p$ & $\alpha$ & $k$ & $d$ & $o_d(p)$ & $i$ & $j$  & $\beta$ & $\mathcal N$\\ \hline


3&2
&0&1&1&1&1&1&0\\

&&0&1&1&2&0&0&0\\

&&0&1&1&2&1&2&0\\

&&0&2&1&1&1&1&0\\

&&0&2&1&2&0&0&0\\

&&0&2&1&2&1&2&0\\

&&0&4&2&1&0&0&0\\

&&0&4&2&1&1&2&0\\

&&0&8&2&1&0&0&0\\

&&0&8&2&1&1&2&1\\

&&1&1&1&2&0&0&0\\

&&1&2&1&2&0&0&0\\

&&1&4&2&1&0&0&0\\

&&1&8&2&1&0&0&1\\

&&2&1&1&2&0&0&0\\

&&2&2&1&2&0&0&4\\

&&3&1&1&1&1&1&0\\

&&3&1&1&2&0&0&72\\

&&3&2&1&1&1&1&1\\

&&3&2&1&2&0&0&0\\

&&4&1&1&2&0&0&72\\

&&4&2&1&2&0&0&4\\

&&4&4&2&1&0&0&2\\ \hline


3&3
&0&1&1&1&1&1&0\\

&&0&1&1&1&2&2&0\\

&&0&1&1&3&0&0&0\\

&&0&1&1&3&1&3&0\\

&&0&2&1&1&1&1&0\\

&&0&2&1&1&2&2&0\\

&&0&2&1&3&0&0&0\\

&&0&2&1&3&1&3&0\\

&&0&13&3&1&0&0&0\\

&&0&13&3&1&1&3&0\\

&&0&26&3&1&0&0&0\\

&&0&26&3&1&1&3&1\\

&&1&1&1&3&0&0&0\\

&&1&2&1&3&0&0&0\\

&&1&13&3&1&0&0&0\\

&&1&26&3&1&0&0&1\\

&&2&1&1&3&0&0&0\\

&&2&2&1&3&0&0&13\\

&&3&1&1&1&1&1&0\\

&&3&1&1&3&0&0&2808\\

&&3&2&1&1&1&1&1\\

\hline
\end{tabular}
\]
\end{table}

\setcounter{table}{0}
\begin{table}
\caption{Continued}
\[
\begin{tabular}{ll|llllll|l}
$p$ & $\alpha$ & $k$ & $d$ & $o_d(p)$ & $i$ & $j$  & $\beta$ & $\mathcal N$\\ \hline

3&3
&3&2&1&3&0&0&0\\

&&4&1&1&3&0&0&15444\\

&&4&2&1&3&0&0&78\\

&&5&1&1&3&0&0&78624\\

&&5&2&1&3&0&0&78\\

&&6&1&1&1&1&1&0\\

&&6&1&1&3&0&0&289224\\

&&6&2&1&1&1&1&4\\

&&6&2&1&3&0&0&234\\

&&7&1&1&3&0&0&880308\\

&&7&2&1&3&0&0&286\\

&&8&1&1&3&0&0&2200770\\

&&8&2&1&3&0&0&715\\

&&9&1&1&1&1&1&72\\

&&9&1&1&1&2&2&0\\

&&9&1&1&3&0&0&4666896\\

&&9&2&1&1&1&1&0\\

&&9&2&1&1&2&2&1\\

&&9&2&1&3&0&0&702\\

&&10&1&1&3&0&0&8401536\\

&&10&2&1&3&0&0&1287\\

&&11&1&1&3&0&0&13003146\\

&&11&2&1&3&0&0&1287\\

&&12&1&1&1&1&1&72\\

&&12&1&1&3&0&0&17337996\\

&&12&2&1&1&1&1&6\\

&&12&2&1&3&0&0&1638\\

&&13&1&1&3&0&0&20011914\\

&&13&2&1&3&0&0&1716\\

&&13&13&3&1&0&0&2\\ \hline


3&4
&0&1&1&1&1&1&0\\

&&0&1&1&1&2&2&0\\

&&0&1&1&1&3&3&0\\

&&0&1&1&2&1&2&0\\

&&0&1&1&4&0&0&0\\

&&0&1&1&4&1&4&0\\

&&0&2&1&1&1&1&0\\

&&0&2&1&1&2&2&0\\

&&0&2&1&1&3&3&0\\

&&0&2&1&2&1&2&0\\

&&0&2&1&4&0&0&0\\

&&0&2&1&4&1&4&0\\

&&0&4&2&1&1&2&0\\

&&0&4&2&2&0&0&0\\

\hline
\end{tabular}
\]
\end{table}

\setcounter{table}{0}
\begin{table}
\caption{Continued}
\[
\begin{tabular}{ll|llllll|l}
$p$ & $\alpha$ & $k$ & $d$ & $o_d(p)$ & $i$ & $j$  & $\beta$ & $\mathcal N$\\ \hline

3&4
&0&4&2&2&1&4&0\\

&&0&5&4&1&0&0&0\\

&&0&5&4&1&1&4&0\\

&&0&8&2&1&1&2&0\\

&&0&8&2&2&0&0&0\\

&&0&8&2&2&1&4&0\\

&&0&10&4&1&0&0&0\\

&&0&10&4&1&1&4&0\\

&&0&16&4&1&0&0&0\\

&&0&16&4&1&1&4&0\\

&&0&20&4&1&0&0&0\\

&&0&20&4&1&1&4&0\\

&&0&40&4&1&0&0&0\\

&&0&40&4&1&1&4&0\\

&&0&80&4&1&0&0&0\\

&&0&80&4&1&1&4&1\\

&&1&1&1&4&0&0&0\\

&&1&2&1&4&0&0&0\\

&&1&4&2&2&0&0&0\\

&&1&5&4&1&0&0&0\\

&&1&8&2&2&0&0&0\\

&&1&10&4&1&0&0&0\\

&&1&16&4&1&0&0&0\\

&&1&20&4&1&0&0&0\\

&&1&40&4&1&0&0&0\\

&&1&80&4&1&0&0&1\\

&&2&1&1&4&0&0&0\\

&&2&2&1&4&0&0&40\\

&&3&1&1&1&1&1&0\\

&&3&1&1&4&0&0&84240\\

&&3&2&1&1&1&1&1\\

&&3&2&1&4&0&0&0\\

&&4&1&1&4&0&0&1600560\\

&&4&2&1&4&0&0&760\\

&&4&4&2&2&0&0&20\\

&&5&1&1&4&0&0&25557120\\

&&5&2&1&4&0&0&760\\

&&5&4&2&2&0&0&20\\

&&5&5&4&1&0&0&16\\

&&6&1&1&1&1&1&0\\

&&6&1&1&4&0&0&323766720\\

&&6&2&1&1&1&1&13\\

&&6&2&1&4&0&0&9360\\

&&6&5&4&1&0&0&16\\

\hline
\end{tabular}
\]
\end{table}

\setcounter{table}{0}
\begin{table}
\caption{Continued}
\[
\begin{tabular}{ll|llllll|l}
$p$ & $\alpha$ & $k$ & $d$ & $o_d(p)$ & $i$ & $j$  & $\beta$ & $\mathcal N$\\ \hline

3&4
&7&1&1&4&0&0&3476416320\\

&&7&2&1&4&0&0&9880\\

&&8&1&1&4&0&0&32156850960\\

&&8&2&1&4&0&0&91200\\

&&8&4&2&2&0&0&180\\

&&8&8&2&2&0&0&10\\

&&9&1&1&1&1&1&2808\\

&&9&1&1&1&2&2&0\\

&&9&1&1&2&1&2&0\\

&&9&1&1&4&0&0&260880328800\\

&&9&2&1&1&1&1&0\\

&&9&2&1&1&2&2&1\\

&&9&2&1&2&1&2&0\\

&&9&2&1&4&0&0&91080\\

&&9&4&2&1&1&2&0\\

&&9&4&2&2&0&0&180\\

&&9&8&2&1&1&2&1\\

&&9&8&2&2&0&0&0\\

&&10&1&1&4&0&0&1878339099600\\

&&10&2&1&4&0&0&658000\\

&&10&5&4&1&0&0&112\\

&&10&10&4&1&0&0&8\\

&&11&1&1&4&0&0&12124115866800\\

&&11&2&1&4&0&0&658000\\

&&11&5&4&1&0&0&112\\

&&11&10&4&1&0&0&8\\

&&12&1&1&1&1&1&15444\\

&&12&1&1&4&0&0&70724008826640\\

&&12&2&1&1&1&1&78\\

&&12&2&1&4&0&0&3834120\\

&&12&4&2&2&0&0&1140\\

&&13&1&1&4&0&0&375382619302320\\

&&13&2&1&4&0&0&3837240\\

&&13&4&2&2&0&0&1140\\

&&14&1&1&4&0&0&1823287008039840\\

&&14&2&1&4&0&0&18643560\\

&&15&1&1&1&1&1&78624\\

&&15&1&1&4&0&0&8144020534667760\\

&&15&2&1&1&1&1&78\\

&&15&2&1&4&0&0&18640440\\

&&15&5&4&1&0&0&560\\

&&16&1&1&4&0&0&33594084717924240\\

&&16&2&1&4&0&0&76899840\\

&&16&4&2&2&0&0&4800\\

\hline
\end{tabular}
\]
\end{table}

\setcounter{table}{0}
\begin{table}
\caption{Continued}
\[
\begin{tabular}{ll|llllll|l}
$p$ & $\alpha$ & $k$ & $d$ & $o_d(p)$ & $i$ & $j$  & $\beta$ & $\mathcal N$\\ \hline

3&4
&16&5&4&1&0&0&560\\

&&16&8&2&2&0&0&40\\

&&16&16&4&1&0&0&5\\

&&17&1&1&4&0&0&128447988569025840\\

&&17&2&1&4&0&0&76899840\\

&&17&4&2&2&0&0&4800\\

&&17&8&2&2&0&0&40\\

&&17&16&4&1&0&0&5\\

&&18&1&1&1&1&1&289224\\

&&18&1&1&1&2&2&0\\

&&18&1&1&2&1&2&0\\

&&18&1&1&4&0&0&456703959345510240\\

&&18&2&1&1&1&1&234\\

&&18&2&1&1&2&2&4\\

&&18&2&1&2&1&2&4\\

&&18&2&1&4&0&0&273429000\\

&&19&1&1&4&0&0&1514334232315681920\\

&&19&2&1&4&0&0&273438880\\

&&20&1&1&4&0&0&4694436120178468800\\

&&20&2&1&4&0&0&847645000\\

&&20&4&2&2&0&0&15500\\

&&20&5&4&1&0&0&1792\\

&&20&10&4&1&0&0&24\\

&&20&20&4&1&0&0&4\\

&&21&1&1&1&1&1&880308\\

&&21&1&1&4&0&0&13636219336980287040\\

&&21&2&1&1&1&1&286\\

&&21&2&1&4&0&0&847633560\\

&&21&4&2&2&0&0&15500\\

&&21&5&4&1&0&0&1792\\

&&21&10&4&1&0&0&24\\

&&21&20&4&1&0&0&4\\

&&22&1&1&4&0&0&37189689100950072960\\

&&22&2&1&4&0&0&2311801440\\

&&23&1&1&4&0&0&95399637552055100160\\

&&23&2&1&4&0&0&2311801440\\

&&24&1&1&1&1&1&2200770\\

&&24&1&1&4&0&0&230549124084046672320\\

&&24&2&1&1&1&1&715\\

&&24&2&1&4&0&0&5586786120\\

&&24&4&2&2&0&0&38640\\

&&24&8&2&2&0&0&120\\

&&25&1&1&4&0&0&525652003491068216880\\

&&25&2&1&4&0&0&5586814720\\

\hline
\end{tabular}
\]
\end{table}

\setcounter{table}{0}
\begin{table}
\caption{Continued}
\[
\begin{tabular}{ll|llllll|l}
$p$ & $\alpha$ & $k$ & $d$ & $o_d(p)$ & $i$ & $j$  & $\beta$ & $\mathcal N$\\ \hline

3&4
&25&4&2&2&0&0&38640\\

&&25&5&4&1&0&0&4368\\

&&25&8&2&2&0&0&120\\

&&26&1&1&4&0&0&1132173545980762721520\\

&&26&2&1&4&0&0&12033222880\\

&&26&5&4&1&0&0&4368\\

&&27&1&1&1&1&1&4666896\\

&&27&1&1&1&2&2&72\\

&&27&1&1&1&3&3&0\\

&&27&1&1&2&1&2&72\\

&&27&1&1&4&0&0&2306279446526974783680\\

&&27&2&1&1&1&1&702\\

&&27&2&1&1&2&2&0\\

&&27&2&1&1&3&3&1\\

&&27&2&1&2&1&2&0\\

&&27&2&1&4&0&0&12033194760\\

&&28&1&1&4&0&0&4447824646873808473920\\

&&28&2&1&4&0&0&23206852320\\

&&28&4&2&2&0&0&77520\\

&&29&1&1&4&0&0&8128782976876762783680\\

&&29&2&1&4&0&0&23206852320\\

&&29&4&2&2&0&0&77520\\

&&30&1&1&1&1&1&8401536\\

&&30&1&1&4&0&0&14089890493252721566080\\

&&30&2&1&1&1&1&1287\\

&&30&2&1&4&0&0&40225293520\\

&&30&5&4&1&0&0&7952\\

&&30&10&4&1&0&0&56\\

&&31&1&1&4&0&0&23180142426486157067760\\

&&31&2&1&4&0&0&40225345000\\

&&31&5&4&1&0&0&7952\\

&&31&10&4&1&0&0&56\\

&&32&1&1&4&0&0&36218972541384621424800\\

&&32&2&1&4&0&0&62851975680\\

&&32&4&2&2&0&0&125760\\

&&32&8&2&2&0&0&200\\

&&32&16&4&1&0&0&10\\

&&33&1&1&1&1&1&13003146\\

&&33&1&1&4&0&0&53779686503311687913280\\

&&33&2&1&1&1&1&1287\\

&&33&2&1&4&0&0&62851924200\\

&&33&4&2&2&0&0&125760\\

&&33&8&2&2&0&0&200\\

&&33&16&4&1&0&0&10\\

\hline
\end{tabular}
\]
\end{table}

\setcounter{table}{0}
\begin{table}
\caption{Continued}
\[
\begin{tabular}{ll|llllll|l}
$p$ & $\alpha$ & $k$ & $d$ & $o_d(p)$ & $i$ & $j$  & $\beta$ & $\mathcal N$\\ \hline

3&4
&34&1&1&4&0&0&75924263298793701542400\\

&&34&2&1&4&0&0&88732378800\\

&&35&1&1&4&0&0&101955439289415766064400\\

&&35&2&1&4&0&0&88732378800\\

&&35&5&4&1&0&0&11440\\

&&36&1&1&1&1&1&17337996\\

&&36&1&1&1&2&2&72\\

&&36&1&1&2&1&2&72\\

&&36&1&1&4&0&0&130276394647586122515840\\

&&36&2&1&1&1&1&1638\\

&&36&2&1&1&2&2&6\\

&&36&2&1&2&1&2&4\\

&&36&2&1&4&0&0&113380027560\\

&&36&4&2&1&1&2&2\\

&&36&4&2&2&0&0&167940\\

&&36&5&4&1&0&0&11440\\

&&37&1&1&4&0&0&158444263762564243557600\\

&&37&2&1&4&0&0&113380093840\\

&&37&4&2&2&0&0&167960\\

&&38&1&1&4&0&0&183461779093495439908800\\

&&38&2&1&4&0&0&131282408400\\

&&39&1&1&1&1&1&20011968\\

&&39&1&1&4&0&0&202278371822123290827840\\

&&39&2&1&1&1&1&1716\\

&&39&2&1&4&0&0&131282339760\\

&&40&1&1&4&0&0&212392290413230290943200\\

&&40&2&1&4&0&0&137846344000\\

&&40&4&2&2&0&0&184500\\

&&40&5&4&1&0&0&12800\\

&&40&8&2&2&0&0&250\\

&&40&10&4&1&0&0&64\\

&&40&20&4&1&0&0&4\\

&&40&40&4&1&0&0&2\\ \hline


5&1
&0&1&1&1&0&0&0\\

&&0&1&1&1&1&1&0\\

&&0&2&1&1&0&0&0\\

&&0&2&1&1&1&1&0\\

&&0&4&1&1&0&0&0\\

&&0&4&1&1&1&1&1\\

&&1&1&1&1&0&0&0\\

&&1&2&1&1&0&0&0\\

&&1&4&1&1&0&0&1\\

&&2&1&1&1&0&0&0\\

&&2&2&1&1&0&0&2\\ \hline

\end{tabular}
\]
\end{table}

\setcounter{table}{0}
\begin{table}
\caption{Continued}
\[
\begin{tabular}{ll|llllll|l}
$p$ & $\alpha$ & $k$ & $d$ & $o_d(p)$ & $i$ & $j$  & $\beta$ & $\mathcal N$\\ \hline


5&2
&0&1&1&1&1&1&0\\

&&0&1&1&2&0&0&0\\

&&0&1&1&2&1&2&0\\

&&0&2&1&1&1&1&0\\

&&0&2&1&2&0&0&0\\

&&0&2&1&2&1&2&0\\

&&0&3&2&1&0&0&0\\

&&0&3&2&1&1&2&0\\

&&0&4&1&1&1&1&0\\

&&0&4&1&2&0&0&0\\

&&0&4&1&2&1&2&0\\

&&0&6&2&1&0&0&0\\

&&0&6&2&1&1&2&0\\

&&0&8&2&1&0&0&0\\

&&0&8&2&1&1&2&0\\

&&0&12&2&1&0&0&0\\

&&0&12&2&1&1&2&0\\

&&0&24&2&1&0&0&0\\

&&0&24&2&1&1&2&1\\

&&1&1&1&2&0&0&0\\

&&1&2&1&2&0&0&0\\

&&1&3&2&1&0&0&0\\

&&1&4&1&2&0&0&0\\

&&1&6&2&1&0&0&0\\

&&1&8&2&1&0&0&0\\

&&1&12&2&1&0&0&0\\

&&1&24&2&1&0&0&1\\

&&2&1&1&2&0&0&0\\

&&2&2&1&2&0&0&12\\

&&3&1&1&2&0&0&1800\\

&&3&2&1&2&0&0&12\\

&&3&3&2&1&0&0&8\\

&&4&1&1&2&0&0&10800\\

&&4&2&1&2&0&0&60\\

&&4&3&2&1&0&0&8\\

&&4&4&1&2&0&0&6\\

&&5&1&1&1&1&1&0\\

&&5&1&1&2&0&0&51600\\

&&5&2&1&1&1&1&0\\

&&5&2&1&2&0&0&60\\

&&5&4&1&1&1&1&1\\

&&5&4&1&2&0&0&0\\

&&6&1&1&2&0&0&171000\\

&&6&2&1&2&0&0&216\\

\hline
\end{tabular}
\]
\end{table}

\setcounter{table}{0}
\begin{table}
\caption{Continued}
\vskip-10pt
\[
\begin{tabular}{ll|llllll|l}
$p$ & $\alpha$ & $k$ & $d$ & $o_d(p)$ & $i$ & $j$  & $\beta$ & $\mathcal N$\\ \hline

5&2
&6&3&2&1&0&0&24\\

&&6&6&2&1&0&0&4\\

&&7&1&1&2&0&0&474600\\

&&7&2&1&2&0&0&216\\

&&7&3&2&1&0&0&24\\

&&7&6&2&1&0&0&4\\

&&8&1&1&2&0&0&1069200\\

&&8&2&1&2&0&0&480\\

&&8&4&1&2&0&0&12\\

&&8&8&2&1&0&0&3\\

&&9&1&1&2&0&0&2029200\\

&&9&2&1&2&0&0&480\\

&&9&3&2&1&0&0&56\\

&&9&4&1&2&0&0&12\\

&&9&8&2&1&0&0&3\\

&&10&1&1&1&1&1&0\\

&&10&1&1&2&0&0&3247800\\

&&10&2&1&1&1&1&2\\

&&10&2&1&2&0&0&780\\

&&10&3&2&1&0&0&56\\

&&11&1&1&2&0&0&4437600\\

&&11&2&1&2&0&0&792\\

&&12&1&1&2&0&0&5175600\\

&&12&2&1&2&0&0&900\\

&&12&3&2&1&0&0&64\\

&&12&4&1&2&0&0&18\\

&&12&6&2&1&0&0&4\\

&&12&12&2&1&0&0&2\\ \hline


7&1
&0&1&1&1&0&0&0\\

&&0&1&1&1&1&1&0\\

&&0&2&1&1&0&0&0\\

&&0&2&1&1&1&1&0\\

&&0&3&1&1&0&0&0\\

&&0&3&1&1&1&1&0\\

&&0&6&1&1&0&0&0\\

&&0&6&1&1&1&1&1\\

&&1&1&1&1&0&0&0\\

&&1&2&1&1&0&0&0\\

&&1&3&1&1&0&0&0\\

&&1&6&1&1&0&0&1\\

&&2&1&1&1&0&0&0\\

&&2&2&1&1&0&0&3\\

&&3&1&1&1&0&0&0\\

&&3&2&1&1&0&0&3\\

&&3&3&1&1&0&0&2\\ \hline

\end{tabular}
\]
\end{table}

\setcounter{table}{0}
\begin{table}
\caption{Continued}
\[
\begin{tabular}{ll|llllll|l}
$p$ & $\alpha$ & $k$ & $d$ & $o_d(p)$ & $i$ & $j$  & $\beta$ & $\mathcal N$\\ \hline


7&2
&0&1&1&1&1&1&0\\

&&0&1&1&2&0&0&0\\

&&0&1&1&2&1&2&0\\

&&0&2&1&1&1&1&0\\

&&0&2&1&2&0&0&0\\

&&0&2&1&2&1&2&0\\

&&0&3&1&1&1&1&0\\

&&0&3&1&2&0&0&0\\

&&0&3&1&2&1&2&0\\

&&0&4&2&1&0&0&0\\

&&0&4&2&1&1&2&0\\

&&0&6&1&1&1&1&0\\

&&0&6&1&2&0&0&0\\

&&0&6&1&2&1&2&0\\

&&0&8&2&1&0&0&0\\

&&0&8&2&1&1&2&0\\

&&0&12&2&1&0&0&0\\

&&0&12&2&1&1&2&0\\

&&0&16&2&1&0&0&0\\

&&0&16&2&1&1&2&0\\

&&0&24&2&1&0&0&0\\

&&0&24&2&1&1&2&0\\

&&0&48&2&1&0&0&0\\

&&0&48&2&1&1&2&1\\

&&1&1&1&2&0&0&0\\

&&1&2&1&2&0&0&0\\

&&1&3&1&2&0&0&0\\

&&1&4&2&1&0&0&0\\

&&1&6&1&2&0&0&0\\

&&1&8&2&1&0&0&0\\

&&1&12&2&1&0&0&0\\

&&1&16&2&1&0&0&0\\

&&1&24&2&1&0&0&0\\

&&1&48&2&1&0&0&1\\

&&2&1&1&2&0&0&0\\

&&2&2&1&2&0&0&24\\

&&3&1&1&2&0&0&16464\\

&&3&2&1&2&0&0&24\\

&&3&3&1&2&0&0&16\\

&&4&1&1&2&0&0&197568\\

&&4&2&1&2&0&0&264\\

&&4&3&1&2&0&0&16\\

&&4&4&2&1&0&0&12\\

\hline
\end{tabular}
\]
\end{table}

\setcounter{table}{0}
\begin{table}
\caption{Continued}
\vskip-10pt
\[
\begin{tabular}{ll|llllll|l}
$p$ & $\alpha$ & $k$ & $d$ & $o_d(p)$ & $i$ & $j$  & $\beta$ & $\mathcal N$\\ \hline

7&2
&5&1&1&2&0&0&1893360\\

&&5&2&1&2&0&0&264\\

&&5&4&2&1&0&0&12\\

&&6&1&1&2&0&0&13879152\\

&&6&2&1&2&0&0&2016\\

&&6&3&1&2&0&0&112\\

&&6&6&1&2&0&0&8\\

&&7&1&1&1&1&1&0\\

&&7&1&1&2&0&0&85796256\\

&&7&2&1&1&1&1&0\\

&&7&2&1&2&0&0&2016\\

&&7&3&1&1&1&1&0\\

&&7&3&1&2&0&0&112\\

&&7&6&1&1&1&1&1\\

&&7&6&1&2&0&0&0\\

&&8&1&1&2&0&0&450457392\\

&&8&2&1&2&0&0&10560\\

&&8&4&2&1&0&0&60\\

&&8&8&2&1&0&0&6\\

&&9&1&1&2&0&0&2053907520\\

&&9&2&1&2&0&0&10560\\

&&9&3&1&2&0&0&560\\

&&9&4&2&1&0&0&60\\

&&9&8&2&1&0&0&6\\

&&10&1&1&2&0&0&8215712400\\

&&10&2&1&2&0&0&42504\\

&&10&3&1&2&0&0&560\\

&&11&1&1&2&0&0&29133833568\\

&&11&2&1&2&0&0&42504\\

&&12&1&1&2&0&0&92257051824\\

&&12&2&1&2&0&0&134352\\

&&12&3&1&2&0&0&1792\\

&&12&4&2&1&0&0&216\\

&&12&6&1&2&0&0&24\\

&&12&12&2&1&0&0&4\\

&&13&1&1&2&0&0&262590100752\\

&&13&2&1&2&0&0&134352\\

&&13&3&1&2&0&0&1792\\

&&13&4&2&1&0&0&216\\

&&13&6&1&2&0&0&24\\

&&13&12&2&1&0&0&4\\

&&14&1&1&1&1&1&0\\

&&14&1&1&2&0&0&675231914448\\

&&14&2&1&1&1&1&3\\

&&14&2&1&2&0&0&346080\\

\hline
\end{tabular}
\]
\end{table}

\setcounter{table}{0}
\begin{table}
\caption{Continued}
\vskip-10pt
\[
\begin{tabular}{ll|llllll|l}
$p$ & $\alpha$ & $k$ & $d$ & $o_d(p)$ & $i$ & $j$  & $\beta$ & $\mathcal N$\\ \hline

7&2
&15&1&1&2&0&0&1575563529456\\

&&15&2&1&2&0&0&346104\\

&&15&3&1&2&0&0&4368\\

&&16&1&1&2&0&0&3348072740880\\

&&16&2&1&2&0&0&734976\\

&&16&3&1&2&0&0&4368\\

&&16&4&2&1&0&0&480\\

&&16&8&2&1&0&0&12\\

&&16&16&2&1&0&0&3\\

&&17&1&1&2&0&0&6499234360080\\

&&17&2&1&2&0&0&734976\\

&&17&4&2&1&0&0&480\\

&&17&8&2&1&0&0&12\\

&&17&16&2&1&0&0&3\\

&&18&1&1&2&0&0&11554194028272\\

&&18&2&1&2&0&0&1307448\\

&&18&3&1&2&0&0&7952\\

&&18&6&1&2&0&0&56\\

&&19&1&1&2&0&0&18851620440240\\

&&19&2&1&2&0&0&1307448\\

&&19&3&1&2&0&0&7952\\

&&19&6&1&2&0&0&56\\

&&20&1&1&2&0&0&28277431244832\\

&&20&2&1&2&0&0&1960464\\

&&20&4&2&1&0&0&792\\

&&21&1&1&1&1&1&0\\

&&21&1&1&2&0&0&39049822056000\\

&&21&2&1&1&1&1&3\\

&&21&2&1&2&0&0&1960440\\

&&21&3&1&1&1&1&2\\

&&21&3&1&2&0&0&11424\\

&&21&4&2&1&0&0&792\\

&&22&1&1&2&0&0&49699773676560\\

&&22&2&1&2&0&0&2496144\\

&&22&3&1&2&0&0&11440\\

&&23&1&1&2&0&0&58343234506368\\

&&23&2&1&2&0&0&2496144\\

&&24&1&1&2&0&0&63205170088032\\

&&24&2&1&2&0&0&2703168\\

&&24&3&1&2&0&0&12800\\

&&24&4&2&1&0&0&900\\

&&24&6&1&2&0&0&64\\

&&24&8&2&1&0&0&18\\

&&24&12&2&1&0&0&4\\

&&24&24&2&1&0&0&2\\ \hline

\end{tabular}
\]
\end{table}

\clearpage

\setcounter{table}{0}
\begin{table}
\caption{Continued}
\[
\begin{tabular}{ll|llllll|l}
$p$ & $\alpha$ & $k$ & $d$ & $o_d(p)$ & $i$ & $j$  & $\beta$ & $\mathcal N$\\ \hline


11&1
&0&1&1&1&0&0&0\\

&&0&1&1&1&1&1&0\\

&&0&2&1&1&0&0&0\\

&&0&2&1&1&1&1&0\\

&&0&5&1&1&0&0&0\\

&&0&5&1&1&1&1&0\\

&&0&10&1&1&0&0&0\\

&&0&10&1&1&1&1&1\\

&&1&1&1&1&0&0&0\\

&&1&2&1&1&0&0&0\\

&&1&5&1&1&0&0&0\\

&&1&10&1&1&0&0&1\\

&&2&1&1&1&0&0&0\\

&&2&2&1&1&0&0&5\\

&&3&1&1&1&0&0&110\\

&&3&2&1&1&0&0&5\\

&&4&1&1&1&0&0&220\\

&&4&2&1&1&0&0&10\\

&&5&1&1&1&0&0&330\\

&&5&2&1&1&0&0&10\\

&&5&5&1&1&0&0&2\\ \hline


13&1
&0&1&1&1&0&0&0\\

&&0&1&1&1&1&1&0\\

&&0&2&1&1&0&0&0\\

&&0&2&1&1&1&1&0\\

&&0&3&1&1&0&0&0\\

&&0&3&1&1&1&1&0\\

&&0&4&1&1&0&0&0\\

&&0&4&1&1&1&1&0\\

&&0&6&1&1&0&0&0\\

&&0&6&1&1&1&1&0\\

&&0&12&1&1&0&0&0\\

&&0&12&1&1&1&1&1\\

&&1&1&1&1&0&0&0\\

&&1&2&1&1&0&0&0\\

&&1&3&1&1&0&0&0\\

&&1&4&1&1&0&0&0\\

&&1&6&1&1&0&0&0\\

&&1&12&1&1&0&0&1\\

&&2&1&1&1&0&0&0\\

&&2&2&1&1&0&0&6\\

&&3&1&1&1&0&0&156\\

&&3&2&1&1&0&0&6\\

&&3&3&1&1&0&0&4\\

\hline
\end{tabular}
\]
\end{table}
\setcounter{table}{0}
\begin{table}
\caption{Continued}
\[
\begin{tabular}{ll|llllll|l}
$p$ & $\alpha$ & $k$ & $d$ & $o_d(p)$ & $i$ & $j$  & $\beta$ & $\mathcal N$\\ \hline

13&1
&4&1&1&1&0&0&468\\

&&4&2&1&1&0&0&12\\

&&4&3&1&1&0&0&4\\

&&4&4&1&1&0&0&3\\

&&5&1&1&1&0&0&1092\\

&&5&2&1&1&0&0&12\\

&&5&4&1&1&0&0&3\\

&&6&1&1&1&0&0&1404\\

&&6&2&1&1&0&0&18\\

&&6&3&1&1&0&0&4\\

&&6&6&1&1&0&0&2\\ \hline


17&1
&0&1&1&1&0&0&0\\

&&0&1&1&1&1&1&0\\

&&0&2&1&1&0&0&0\\

&&0&2&1&1&1&1&0\\

&&0&4&1&1&0&0&0\\

&&0&4&1&1&1&1&0\\

&&0&8&1&1&0&0&0\\

&&0&8&1&1&1&1&0\\

&&0&16&1&1&0&0&0\\

&&0&16&1&1&1&1&1\\

&&1&1&1&1&0&0&0\\

&&1&2&1&1&0&0&0\\

&&1&4&1&1&0&0&0\\

&&1&8&1&1&0&0&0\\

&&1&16&1&1&0&0&1\\

&&2&1&1&1&0&0&0\\

&&2&2&1&1&0&0&8\\

&&3&1&1&1&0&0&544\\

&&3&2&1&1&0&0&8\\

&&4&1&1&1&0&0&1904\\

&&4&2&1&1&0&0&24\\

&&4&4&1&1&0&0&4\\

&&5&1&1&1&0&0&5712\\

&&5&2&1&1&0&0&24\\

&&5&4&1&1&0&0&4\\

&&6&1&1&1&0&0&11424\\

&&6&2&1&1&0&0&56\\

&&7&1&1&1&0&0&18496\\

&&7&2&1&1&0&0&56\\

&&8&1&1&1&0&0&23120\\

&&8&2&1&1&0&0&64\\

&&8&4&1&1&0&0&4\\

&&8&8&1&1&0&0&2\\ \hline

\end{tabular}
\]
\end{table}
\setcounter{table}{0}
\begin{table}
\caption{Continued}
\[
\begin{tabular}{ll|llllll|l}
$p$ & $\alpha$ & $k$ & $d$ & $o_d(p)$ & $i$ & $j$  & $\beta$ & $\mathcal N$\\ \hline


19&1
&0&1&1&1&0&0&0\\

&&0&1&1&1&1&1&0\\

&&0&2&1&1&0&0&0\\

&&0&2&1&1&1&1&0\\

&&0&3&1&1&0&0&0\\

&&0&3&1&1&1&1&0\\

&&0&6&1&1&0&0&0\\

&&0&6&1&1&1&1&0\\

&&0&9&1&1&0&0&0\\

&&0&9&1&1&1&1&0\\

&&0&18&1&1&0&0&0\\

&&0&18&1&1&1&1&1\\

&&1&1&1&1&0&0&0\\

&&1&2&1&1&0&0&0\\

&&1&3&1&1&0&0&0\\

&&1&6&1&1&0&0&0\\

&&1&9&1&1&0&0&0\\

&&1&18&1&1&0&0&1\\

&&2&1&1&1&0&0&0\\

&&2&2&1&1&0&0&9\\

&&3&1&1&1&0&0&684\\

&&3&2&1&1&0&0&9\\

&&3&3&1&1&0&0&6\\

&&4&1&1&1&0&0&3078\\

&&4&2&1&1&0&0&36\\

&&4&3&1&1&0&0&6\\

&&5&1&1&1&0&0&10944\\

&&5&2&1&1&0&0&36\\

&&6&1&1&1&0&0&25308\\

&&6&2&1&1&0&0&81\\

&&6&3&1&1&0&0&12\\

&&6&6&1&1&0&0&3\\

&&7&1&1&1&0&0&48564\\

&&7&2&1&1&0&0&81\\

&&7&3&1&1&0&0&12\\

&&7&6&1&1&0&0&3\\

&&8&1&1&1&0&0&73188\\

&&8&2&1&1&0&0&126\\

&&9&1&1&1&0&0&89604\\

&&9&2&1&1&0&0&126\\

&&9&3&1&1&0&0&18\\

&&9&9&1&1&0&0&2\\ \hline

\end{tabular}
\]
\end{table}
\setcounter{table}{0}
\begin{table}
\caption{Continued}
\[
\begin{tabular}{ll|llllll|l}
$p$ & $\alpha$ & $k$ & $d$ & $o_d(p)$ & $i$ & $j$  & $\beta$ & $\mathcal N$\\ \hline


23&1
&0&1&1&1&0&0&0\\

&&0&1&1&1&1&1&0\\

&&0&2&1&1&0&0&0\\

&&0&2&1&1&1&1&0\\

&&0&11&1&1&0&0&0\\

&&0&11&1&1&1&1&0\\

&&0&22&1&1&0&0&0\\

&&0&22&1&1&1&1&1\\

&&1&1&1&1&0&0&0\\

&&1&2&1&1&0&0&0\\

&&1&11&1&1&0&0&0\\

&&1&22&1&1&0&0&1\\

&&2&1&1&1&0&0&0\\

&&2&2&1&1&0&0&11\\

&&3&1&1&1&0&0&1518\\

&&3&2&1&1&0&0&11\\

&&4&1&1&1&0&0&7590\\

&&4&2&1&1&0&0&55\\

&&5&1&1&1&0&0&32384\\

&&5&2&1&1&0&0&55\\

&&6&1&1&1&0&0&97152\\

&&6&2&1&1&0&0&165\\

&&7&1&1&1&0&0&241362\\

&&7&2&1&1&0&0&165\\

&&8&1&1&1&0&0&482724\\

&&8&2&1&1&0&0&330\\

&&9&1&1&1&0&0&809600\\

&&9&2&1&1&0&0&330\\

&&10&1&1&1&0&0&1133440\\

&&10&2&1&1&0&0&462\\

&&11&1&1&1&0&0&1341406\\

&&11&2&1&1&0&0&462\\

&&11&11&1&1&0&0&2\\ \hline


29&1
&0&1&1&1&0&0&0\\

&&0&1&1&1&1&1&0\\

&&0&2&1&1&0&0&0\\

&&0&2&1&1&1&1&0\\

&&0&4&1&1&0&0&0\\

&&0&4&1&1&1&1&0\\

&&0&7&1&1&0&0&0\\

&&0&7&1&1&1&1&0\\

&&0&14&1&1&0&0&0\\

&&0&14&1&1&1&1&0\\

\hline
\end{tabular}
\]
\end{table}
\setcounter{table}{0}
\begin{table}
\caption{Continued}
\[
\begin{tabular}{ll|llllll|l}
$p$ & $\alpha$ & $k$ & $d$ & $o_d(p)$ & $i$ & $j$  & $\beta$ & $\mathcal N$\\ \hline

29&1
&0&28&1&1&0&0&0\\

&&0&28&1&1&1&1&1\\

&&1&1&1&1&0&0&0\\

&&1&2&1&1&0&0&0\\

&&1&4&1&1&0&0&0\\

&&1&7&1&1&0&0&0\\

&&1&14&1&1&0&0&0\\

&&1&28&1&1&0&0&1\\

&&2&1&1&1&0&0&0\\

&&2&2&1&1&0&0&14\\

&&3&1&1&1&0&0&3248\\

&&3&2&1&1&0&0&14\\

&&4&1&1&1&0&0&21112\\

&&4&2&1&1&0&0&84\\

&&4&4&1&1&0&0&7\\

&&5&1&1&1&0&0&116116\\

&&5&2&1&1&0&0&84\\

&&5&4&1&1&0&0&7\\

&&6&1&1&1&0&0&464464\\

&&6&2&1&1&0&0&364\\

&&7&1&1&1&0&0&1550108\\

&&7&2&1&1&0&0&364\\

&&7&7&1&1&0&0&4\\

&&8&1&1&1&0&0&4263000\\

&&8&2&1&1&0&0&980\\

&&8&4&1&1&0&0&21\\

&&8&7&1&1&0&0&4\\

&&9&1&1&1&0&0&9985976\\

&&9&2&1&1&0&0&980\\

&&9&4&1&1&0&0&21\\

&&10&1&1&1&0&0&19971952\\

&&10&2&1&1&0&0&2002\\

&&11&1&1&1&0&0&34539232\\

&&11&2&1&1&0&0&2002\\

&&12&1&1&1&0&0&51808848\\

&&12&2&1&1&0&0&2968\\

&&12&4&1&1&0&0&35\\

&&13&1&1&1&0&0&67776828\\

&&13&2&1&1&0&0&2968\\

&&13&4&1&1&0&0&35\\

&&14&1&1&1&0&0&77459116\\

&&14&2&1&1&0&0&3430\\

&&14&7&1&1&0&0&4\\

&&14&14&1&1&0&0&2\\ \hline

\end{tabular}
\]
\end{table}
\setcounter{table}{0}
\begin{table}
\caption{Continued}
\[
\begin{tabular}{ll|llllll|l}
$p$ & $\alpha$ & $k$ & $d$ & $o_d(p)$ & $i$ & $j$  & $\beta$ & $\mathcal N$\\ \hline

31&1
&0&1&1&1&0&0&0\\

&&0&1&1&1&1&1&0\\

&&0&2&1&1&0&0&0\\

&&0&2&1&1&1&1&0\\

&&0&3&1&1&0&0&0\\

&&0&3&1&1&1&1&0\\

&&0&5&1&1&0&0&0\\

&&0&5&1&1&1&1&0\\

&&0&6&1&1&0&0&0\\

&&0&6&1&1&1&1&0\\

&&0&10&1&1&0&0&0\\

&&0&10&1&1&1&1&0\\

&&0&15&1&1&0&0&0\\

&&0&15&1&1&1&1&0\\

&&0&30&1&1&0&0&0\\

&&0&30&1&1&1&1&1\\

&&1&1&1&1&0&0&0\\

&&1&2&1&1&0&0&0\\

&&1&3&1&1&0&0&0\\

&&1&5&1&1&0&0&0\\

&&1&6&1&1&0&0&0\\

&&1&10&1&1&0&0&0\\

&&1&15&1&1&0&0&0\\

&&1&30&1&1&0&0&1\\

&&2&1&1&1&0&0&0\\

&&2&2&1&1&0&0&15\\

&&3&1&1&1&0&0&3720\\

&&3&2&1&1&0&0&15\\

&&3&3&1&1&0&0&10\\

&&4&1&1&1&0&0&27900\\

&&4&2&1&1&0&0&105\\

&&4&3&1&1&0&0&10\\

&&5&1&1&1&0&0&166470\\

&&5&2&1&1&0&0&105\\

&&5&5&1&1&0&0&6\\

&&6&1&1&1&0&0&720750\\

&&6&2&1&1&0&0&450\\

&&6&3&1&1&0&0&40\\

&&6&5&1&1&0&0&6\\

&&6&6&1&1&0&0&5\\

&&7&1&1&1&0&0&2614230\\

&&7&2&1&1&0&0&450\\

&&7&3&1&1&0&0&40\\

&&7&6&1&1&0&0&5\\

\hline
\end{tabular}
\]
\end{table}
\setcounter{table}{0}
\begin{table}
\caption{Continued}
\[
\begin{tabular}{ll|llllll|l}
$p$ & $\alpha$ & $k$ & $d$ & $o_d(p)$ & $i$ & $j$  & $\beta$ & $\mathcal N$\\ \hline

31&1
&8&1&1&1&0&0&7846410\\

&&8&2&1&1&0&0&1365\\

&&9&1&1&1&0&0&20114040\\

&&9&2&1&1&0&0&1365\\

&&9&3&1&1&0&0&120\\

&&10&1&1&1&0&0&44254980\\

&&10&2&1&1&0&0&3000\\

&&10&3&1&1&0&0&120\\

&&10&5&1&1&0&0&12\\

&&10&10&1&1&0&0&3\\

&&11&1&1&1&0&0&84578850\\

&&11&2&1&1&0&0&3000\\

&&11&5&1&1&0&0&12\\

&&11&10&1&1&0&0&3\\

&&12&1&1&1&0&0&140959170\\

&&12&2&1&1&0&0&4995\\

&&12&3&1&1&0&0&200\\

&&12&6&1&1&0&0&10\\

&&13&1&1&1&0&0&206091720\\

&&13&2&1&1&0&0&4995\\

&&13&3&1&1&0&0&200\\

&&13&6&1&1&0&0&10\\

&&14&1&1&1&0&0&264983040\\

&&14&2&1&1&0&0&6435\\

&&15&1&1&1&0&0&300332340\\

&&15&2&1&1&0&0&6435\\

&&15&3&1&1&0&0&250\\

&&15&5&1&1&0&0&18\\

&&15&15&1&1&0&0&2\\ \hline


37&1
&0&1&1&1&0&0&0\\

&&0&1&1&1&1&1&0\\

&&0&2&1&1&0&0&0\\

&&0&2&1&1&1&1&0\\

&&0&3&1&1&0&0&0\\

&&0&3&1&1&1&1&0\\

&&0&4&1&1&0&0&0\\

&&0&4&1&1&1&1&0\\

&&0&6&1&1&0&0&0\\

&&0&6&1&1&1&1&0\\

&&0&9&1&1&0&0&0\\

&&0&9&1&1&1&1&0\\

&&0&12&1&1&0&0&0\\

&&0&12&1&1&1&1&0\\

\hline
\end{tabular}
\]
\end{table}
\setcounter{table}{0}
\begin{table}
\caption{Continued}
\vskip-10pt
\[
\begin{tabular}{ll|llllll|l}
$p$ & $\alpha$ & $k$ & $d$ & $o_d(p)$ & $i$ & $j$  & $\beta$ & $\mathcal N$\\ \hline

37&1
&0&18&1&1&0&0&0\\

&&0&18&1&1&1&1&0\\

&&0&36&1&1&0&0&0\\

&&0&36&1&1&1&1&1\\

&&1&1&1&1&0&0&0\\

&&1&2&1&1&0&0&0\\

&&1&3&1&1&0&0&0\\

&&1&4&1&1&0&0&0\\

&&1&6&1&1&0&0&0\\

&&1&9&1&1&0&0&0\\

&&1&12&1&1&0&0&0\\

&&1&18&1&1&0&0&0\\

&&1&36&1&1&0&0&1\\

&&2&1&1&1&0&0&0\\

&&2&2&1&1&0&0&18\\

&&3&1&1&1&0&0&6660\\

&&3&2&1&1&0&0&18\\

&&3&3&1&1&0&0&12\\

&&4&1&1&1&0&0&59940\\

&&4&2&1&1&0&0&144\\

&&4&3&1&1&0&0&12\\

&&4&4&1&1&0&0&9\\

&&5&1&1&1&0&0&430236\\

&&5&2&1&1&0&0&144\\

&&5&4&1&1&0&0&9\\

&&6&1&1&1&0&0&2292372\\

&&6&2&1&1&0&0&810\\

&&6&3&1&1&0&0&60\\

&&6&6&1&1&0&0&6\\

&&7&1&1&1&0&0&10263060\\

&&7&2&1&1&0&0&810\\

&&7&3&1&1&0&0&60\\

&&7&6&1&1&0&0&6\\

&&8&1&1&1&0&0&38494800\\

&&8&2&1&1&0&0&3024\\

&&8&4&1&1&0&0&36\\

&&9&1&1&1&0&0&124282260\\

&&9&2&1&1&0&0&3024\\

&&9&3&1&1&0&0&216\\

&&9&4&1&1&0&0&36\\

&&9&9&1&1&0&0&4\\

&&10&1&1&1&0&0&348004980\\

&&10&2&1&1&0&0&8568\\

&&10&3&1&1&0&0&216\\

&&10&9&1&1&0&0&4\\

\hline
\end{tabular}
\]
\end{table}
\setcounter{table}{0}
\begin{table}
\caption{Continued}
\[
\begin{tabular}{ll|llllll|l}
$p$ & $\alpha$ & $k$ & $d$ & $o_d(p)$ & $i$ & $j$  & $\beta$ & $\mathcal N$\\ \hline
37&1
&11&1&1&1&0&0&854675136\\

&&11&2&1&1&0&0&8568\\

&&12&1&1&1&0&0&1851778368\\

&&12&2&1&1&0&0&18468\\

&&12&3&1&1&0&0&480\\

&&12&4&1&1&0&0&81\\

&&12&6&1&1&0&0&12\\

&&12&12&1&1&0&0&3\\

&&13&1&1&1&0&0&3561762672\\

&&13&2&1&1&0&0&18468\\

&&13&3&1&1&0&0&480\\

&&13&4&1&1&0&0&81\\

&&13&6&1&1&0&0&12\\

&&13&12&1&1&0&0&3\\

&&14&1&1&1&0&0&6105909312\\

&&14&2&1&1&0&0&31824\\

&&15&1&1&1&0&0&9362992968\\

&&15&2&1&1&0&0&31824\\

&&15&3&1&1&0&0&792\\

&&16&1&1&1&0&0&12874126320\\

&&16&2&1&1&0&0&43632\\

&&16&3&1&1&0&0&792\\

&&16&4&1&1&0&0&126\\

&&17&1&1&1&0&0&15903749664\\

&&17&2&1&1&0&0&43632\\

&&17&4&1&1&0&0&126\\

&&18&1&1&1&0&0&17670799512\\

&&18&2&1&1&0&0&48600\\

&&18&3&1&1&0&0&900\\

&&18&6&1&1&0&0&18\\

&&18&9&1&1&0&0&4\\

&&18&18&1&1&0&0&2\\ \hline


41&1
&0&1&1&1&0&0&0\\

&&0&1&1&1&1&1&0\\

&&0&2&1&1&0&0&0\\

&&0&2&1&1&1&1&0\\

&&0&4&1&1&0&0&0\\

&&0&4&1&1&1&1&0\\

&&0&5&1&1&0&0&0\\

&&0&5&1&1&1&1&0\\

&&0&8&1&1&0&0&0\\

&&0&8&1&1&1&1&0\\

&&0&10&1&1&0&0&0\\

&&0&10&1&1&1&1&0\\

\hline
\end{tabular}
\]
\end{table}
\setcounter{table}{0}
\begin{table}
\caption{Continued}
\[
\begin{tabular}{ll|llllll|l}
$p$ & $\alpha$ & $k$ & $d$ & $o_d(p)$ & $i$ & $j$  & $\beta$ & $\mathcal N$\\ \hline

41&1
&0&20&1&1&0&0&0\\

&&0&20&1&1&1&1&0\\

&&0&40&1&1&0&0&0\\

&&0&40&1&1&1&1&1\\

&&1&1&1&1&0&0&0\\

&&1&2&1&1&0&0&0\\

&&1&4&1&1&0&0&0\\

&&1&5&1&1&0&0&0\\

&&1&8&1&1&0&0&0\\

&&1&10&1&1&0&0&0\\

&&1&20&1&1&0&0&0\\

&&1&40&1&1&0&0&1\\

&&2&1&1&1&0&0&0\\

&&2&2&1&1&0&0&20\\

&&3&1&1&1&0&0&9840\\

&&3&2&1&1&0&0&20\\

&&4&1&1&1&0&0&93480\\

&&4&2&1&1&0&0&180\\

&&4&4&1&1&0&0&10\\

&&5&1&1&1&0&0&741280\\

&&5&2&1&1&0&0&180\\

&&5&4&1&1&0&0&10\\

&&5&5&1&1&0&0&8\\

&&6&1&1&1&0&0&4449320\\

&&6&2&1&1&0&0&1140\\

&&6&5&1&1&0&0&8\\

&&7&1&1&1&0&0&22435200\\

&&7&2&1&1&0&0&1140\\

&&8&1&1&1&0&0&95349600\\

&&8&2&1&1&0&0&4800\\

&&8&4&1&1&0&0&40\\

&&8&8&1&1&0&0&5\\

&&9&1&1&1&0&0&350144920\\

&&9&2&1&1&0&0&4800\\

&&9&4&1&1&0&0&40\\

&&9&8&1&1&0&0&5\\

&&10&1&1&1&0&0&1120462760\\

&&10&2&1&1&0&0&15500\\

&&10&5&1&1&0&0&24\\

&&10&10&1&1&0&0&4\\

&&11&1&1&1&0&0&3158825320\\

&&11&2&1&1&0&0&15500\\

&&11&5&1&1&0&0&24\\

&&11&10&1&1&0&0&4\\

\hline
\end{tabular}
\]
\end{table}
\clearpage
\setcounter{table}{0}
\begin{table}
\caption{Continued}
\[
\begin{tabular}{ll|llllll|l}
$p$ & $\alpha$ & $k$ & $d$ & $o_d(p)$ & $i$ & $j$  & $\beta$ & $\mathcal N$\\ \hline

41&1
&12&1&1&1&0&0&7897065760\\

&&12&2&1&1&0&0&38640\\

&&12&4&1&1&0&0&120\\

&&13&1&1&1&0&0&17618487200\\

&&13&2&1&1&0&0&38640\\

&&13&4&1&1&0&0&120\\

&&14&1&1&1&0&0&35236974400\\

&&14&2&1&1&0&0&77520\\

&&15&1&1&1&0&0&63429094280\\

&&15&2&1&1&0&0&77520\\

&&15&5&1&1&0&0&56\\

&&16&1&1&1&0&0&103072279640\\

&&16&2&1&1&0&0&125760\\

&&16&4&1&1&0&0&200\\

&&16&5&1&1&0&0&56\\

&&16&8&1&1&0&0&10\\

&&17&1&1&1&0&0&151579315680\\

&&17&2&1&1&0&0&125760\\

&&17&4&1&1&0&0&200\\

&&17&8&1&1&0&0&10\\

&&18&1&1&1&0&0&202105754240\\

&&18&2&1&1&0&0&167960\\

&&19&1&1&1&0&0&244655783840\\

&&19&2&1&1&0&0&167960\\

&&20&1&1&1&0&0&269121359600\\

&&20&2&1&1&0&0&184500\\

&&20&4&1&1&0&0&250\\

&&20&5&1&1&0&0&64\\

&&20&10&1&1&0&0&4\\

&&20&20&1&1&0&0&2\\ \hline


43&1
&0&1&1&1&0&0&0\\

&&0&1&1&1&1&1&0\\

&&0&2&1&1&0&0&0\\

&&0&2&1&1&1&1&0\\

&&0&3&1&1&0&0&0\\

&&0&3&1&1&1&1&0\\

&&0&6&1&1&0&0&0\\

&&0&6&1&1&1&1&0\\

&&0&7&1&1&0&0&0\\

&&0&7&1&1&1&1&0\\

&&0&14&1&1&0&0&0\\

&&0&14&1&1&1&1&0\\

&&0&21&1&1&0&0&0\\

&&0&21&1&1&1&1&0\\

\hline
\end{tabular}
\]
\end{table}
\clearpage
\setcounter{table}{0}
\begin{table}
\caption{Continued}
\[
\begin{tabular}{ll|llllll|l}
$p$ & $\alpha$ & $k$ & $d$ & $o_d(p)$ & $i$ & $j$  & $\beta$ & $\mathcal N$\\ \hline

43&1
&0&42&1&1&0&0&0\\

&&0&42&1&1&1&1&1\\

&&1&1&1&1&0&0&0\\

&&1&2&1&1&0&0&0\\

&&1&3&1&1&0&0&0\\

&&1&6&1&1&0&0&0\\

&&1&7&1&1&0&0&0\\

&&1&14&1&1&0&0&0\\

&&1&21&1&1&0&0&0\\

&&1&42&1&1&0&0&1\\

&&2&1&1&1&0&0&0\\

&&2&2&1&1&0&0&21\\

&&3&1&1&1&0&0&10836\\

&&3&2&1&1&0&0&21\\

&&3&3&1&1&0&0&14\\

&&4&1&1&1&0&0&113778\\

&&4&2&1&1&0&0&210\\

&&4&3&1&1&0&0&14\\

&&5&1&1&1&0&0&953568\\

&&5&2&1&1&0&0&210\\

&&6&1&1&1&0&0&6035652\\

&&6&2&1&1&0&0&1323\\

&&6&3&1&1&0&0&84\\

&&6&6&1&1&0&0&7\\

&&7&1&1&1&0&0&32163054\\

&&7&2&1&1&0&0&1323\\

&&7&3&1&1&0&0&84\\

&&7&6&1&1&0&0&7\\

&&7&7&1&1&0&0&6\\

&&8&1&1&1&0&0&144750900\\

&&8&2&1&1&0&0&5985\\

&&8&7&1&1&0&0&6\\

&&9&1&1&1&0&0&563648988\\

&&9&2&1&1&0&0&5985\\

&&9&3&1&1&0&0&364\\

&&10&1&1&1&0&0&1916444124\\

&&10&2&1&1&0&0&20349\\

&&10&3&1&1&0&0&364\\

&&11&1&1&1&0&0&5751129342\\

&&11&2&1&1&0&0&20349\\

&&12&1&1&1&0&0&15336302772\\

&&12&2&1&1&0&0&54243\\

&&12&3&1&1&0&0&980\\

&&12&6&1&1&0&0&21\\

\hline
\end{tabular}
\]
\end{table}
\clearpage
\setcounter{table}{0}
\begin{table}
\caption{Continued}
\vskip-10pt
\[
\begin{tabular}{ll|llllll|l}
$p$ & $\alpha$ & $k$ & $d$ & $o_d(p)$ & $i$ & $j$  & $\beta$ & $\mathcal N$\\ \hline

43&1
&13&1&1&1&0&0&36574472676\\

&&13&2&1&1&0&0&54243\\

&&13&3&1&1&0&0&980\\

&&13&6&1&1&0&0&21\\

&&14&1&1&1&0&0&78373959804\\

&&14&2&1&1&0&0&116277\\

&&14&7&1&1&0&0&12\\

&&14&14&1&1&0&0&3\\

&&15&1&1&1&0&0&151527570054\\

&&15&2&1&1&0&0&116277\\

&&15&3&1&1&0&0&2002\\

&&15&7&1&1&0&0&12\\

&&15&14&1&1&0&0&3\\

&&16&1&1&1&0&0&265173313062\\

&&16&2&1&1&0&0&203490\\

&&16&3&1&1&0&0&2002\\

&&17&1&1&1&0&0&421162898688\\

&&17&2&1&1&0&0&203490\\

&&18&1&1&1&0&0&608346281592\\

&&18&2&1&1&0&0&293895\\

&&18&3&1&1&0&0&2968\\

&&18&6&1&1&0&0&35\\

&&19&1&1&1&0&0&800459665236\\

&&19&2&1&1&0&0&293895\\

&&19&3&1&1&0&0&2968\\

&&19&6&1&1&0&0&35\\

&&20&1&1&1&0&0&960551751432\\

&&20&2&1&1&0&0&352716\\

&&21&1&1&1&0&0&1052034166722\\

&&21&2&1&1&0&0&352716\\

&&21&3&1&1&0&0&3430\\

&&21&7&1&1&0&0&18\\

&&21&21&1&1&0&0&2\\ \hline


47&1
&0&1&1&1&0&0&0\\

&&0&1&1&1&1&1&0\\

&&0&2&1&1&0&0&0\\

&&0&2&1&1&1&1&0\\

&&0&23&1&1&0&0&0\\

&&0&23&1&1&1&1&0\\

&&0&46&1&1&0&0&0\\

&&0&46&1&1&1&1&1\\

&&1&1&1&1&0&0&0\\

&&1&2&1&1&0&0&0\\

&&1&23&1&1&0&0&0\\

&&1&46&1&1&0&0&1\\

\hline
\end{tabular}
\]
\end{table}
\setcounter{table}{0}
\begin{table}
\caption{Continued}
\vskip-10pt
\[
\begin{tabular}{ll|llllll|l}
$p$ & $\alpha$ & $k$ & $d$ & $o_d(p)$ & $i$ & $j$  & $\beta$ & $\mathcal N$\\ \hline

47&1
&2&1&1&1&0&0&0\\

&&2&2&1&1&0&0&23\\

&&3&1&1&1&0&0&15134\\

&&3&2&1&1&0&0&23\\

&&4&1&1&1&0&0&166474\\

&&4&2&1&1&0&0&253\\

&&5&1&1&1&0&0&1522048\\

&&5&2&1&1&0&0&253\\

&&6&1&1&1&0&0&10654336\\

&&6&2&1&1&0&0&1771\\

&&7&1&1&1&0&0&62808262\\

&&7&2&1&1&0&0&1771\\

&&8&1&1&1&0&0&314041310\\

&&8&2&1&1&0&0&8855\\

&&9&1&1&1&0&0&1362232960\\

&&9&2&1&1&0&0&8855\\

&&10&1&1&1&0&0&5176485248\\

&&10&2&1&1&0&0&33649\\

&&11&1&1&1&0&0&17415552114\\

&&11&2&1&1&0&0&33649\\

&&12&1&1&1&0&0&52246656342\\

&&12&2&1&1&0&0&100947\\

&&13&1&1&1&0&0&140672103936\\

&&13&2&1&1&0&0&100947\\

&&14&1&1&1&0&0&341632252416\\

&&14&2&1&1&0&0&245157\\

&&15&1&1&1&0&0&751604782170\\

&&15&2&1&1&0&0&245157\\

&&16&1&1&1&0&0&1503209564340\\

&&16&2&1&1&0&0&490314\\

&&17&1&1&1&0&0&2741165830656\\

&&17&2&1&1&0&0&490314\\

&&18&1&1&1&0&0&4568609717760\\

&&18&2&1&1&0&0&817190\\

&&19&1&1&1&0&0&6973161362860\\

&&19&2&1&1&0&0&817190\\

&&20&1&1&1&0&0&9762425908004\\

&&20&2&1&1&0&0&1144066\\

&&21&1&1&1&0&0&12551705816320\\

&&21&2&1&1&0&0&1144066\\

&&22&1&1&1&0&0&14833834146560\\

&&22&2&1&1&0&0&1352078\\

&&23&1&1&1&0&0&16123738293790\\

&&23&2&1&1&0&0&1352078\\

&&23&23&1&1&0&0&2\\ \hline

\end{tabular}
\]
\end{table}
\setcounter{table}{0}
\begin{table}
\caption{Continued}
\[
\begin{tabular}{ll|llllll|l}
$p$ & $\alpha$ & $k$ & $d$ & $o_d(p)$ & $i$ & $j$  & $\beta$ & $\mathcal N$\\ \hline


53&1
&0&1&1&1&0&0&0\\

&&0&1&1&1&1&1&0\\

&&0&2&1&1&0&0&0\\

&&0&2&1&1&1&1&0\\

&&0&4&1&1&0&0&0\\

&&0&4&1&1&1&1&0\\

&&0&13&1&1&0&0&0\\

&&0&13&1&1&1&1&0\\

&&0&26&1&1&0&0&0\\

&&0&26&1&1&1&1&0\\

&&0&52&1&1&0&0&0\\

&&0&52&1&1&1&1&1\\

&&1&1&1&1&0&0&0\\

&&1&2&1&1&0&0&0\\

&&1&4&1&1&0&0&0\\

&&1&13&1&1&0&0&0\\

&&1&26&1&1&0&0&0\\

&&1&52&1&1&0&0&1\\

&&2&1&1&1&0&0&0\\

&&2&2&1&1&0&0&26\\

&&3&1&1&1&0&0&22048\\

&&3&2&1&1&0&0&26\\

&&4&1&1&1&0&0&275600\\

&&4&2&1&1&0&0&312\\

&&4&4&1&1&0&0&13\\

&&5&1&1&1&0&0&2852460\\

&&5&2&1&1&0&0&312\\

&&5&4&1&1&0&0&13\\

&&6&1&1&1&0&0&22819680\\

&&6&2&1&1&0&0&2600\\

&&7&1&1&1&0&0&154005280\\

&&7&2&1&1&0&0&2600\\

&&8&1&1&1&0&0&885530360\\

&&8&2&1&1&0&0&14872\\

&&8&4&1&1&0&0&78\\

&&9&1&1&1&0&0&4430821200\\

&&9&2&1&1&0&0&14872\\

&&9&4&1&1&0&0&78\\

&&10&1&1&1&0&0&19495613280\\

&&10&2&1&1&0&0&65780\\

&&11&1&1&1&0&0&76220266720\\

&&11&2&1&1&0&0&65780\\

\hline
\end{tabular}
\]
\end{table}
\setcounter{table}{0}
\begin{table}
\caption{Continued}
\vskip-10pt
\[
\begin{tabular}{ll|llllll|l}
$p$ & $\alpha$ & $k$ & $d$ & $o_d(p)$ & $i$ & $j$  & $\beta$ & $\mathcal N$\\ \hline

53&1
&12&1&1&1&0&0&266770933520\\

&&12&2&1&1&0&0&229944\\

&&12&4&1&1&0&0&286\\

&&13&1&1&1&0&0&841380764068\\

&&13&2&1&1&0&0&229944\\

&&13&4&1&1&0&0&286\\

&&13&13&1&1&0&0&4\\

&&14&1&1&1&0&0&2403945040588\\

&&14&2&1&1&0&0&657800\\

&&14&13&1&1&0&0&4\\

&&15&1&1&1&0&0&6250312887520\\

&&15&2&1&1&0&0&657800\\

&&16&1&1&1&0&0&14844493107860\\

&&16&2&1&1&0&0&1561560\\

&&16&4&1&1&0&0&715\\

&&17&1&1&1&0&0&32308700058960\\

&&17&2&1&1&0&0&1561560\\

&&17&4&1&1&0&0&715\\

&&18&1&1&1&0&0&64617400117920\\

&&18&2&1&1&0&0&3124550\\

&&19&1&1&1&0&0&119032192302400\\

&&19&2&1&1&0&0&3124550\\

&&20&1&1&1&0&0&202354726914080\\

&&20&2&1&1&0&0&5310448\\

&&20&4&1&1&0&0&1287\\

&&21&1&1&1&0&0&317986160306100\\

&&21&2&1&1&0&0&5310448\\

&&21&4&1&1&0&0&1287\\

&&22&1&1&1&0&0&462525324081600\\

&&22&2&1&1&0&0&7726160\\

&&23&1&1&1&0&0&623403840105280\\

&&23&2&1&1&0&0&7726160\\

&&24&1&1&1&0&0&779254800131600\\

&&24&2&1&1&0&0&9655984\\

&&24&4&1&1&0&0&1716\\

&&25&1&1&1&0&0&903935650049952\\

&&25&2&1&1&0&0&9655984\\

&&25&4&1&1&0&0&1716\\

&&26&1&1&1&0&0&973469161592044\\

&&26&2&1&1&0&0&10400598\\

&&26&13&1&1&0&0&4\\

&&26&26&1&1&0&0&2\\ \hline

\end{tabular}
\]
\end{table}
\setcounter{table}{0}
\begin{table}
\caption{Continued}
\[
\begin{tabular}{ll|llllll|l}
$p$ & $\alpha$ & $k$ & $d$ & $o_d(p)$ & $i$ & $j$  & $\beta$ & $\mathcal N$\\ \hline


59&1
&0&1&1&1&0&0&0\\

&&0&1&1&1&1&1&0\\

&&0&2&1&1&0&0&0\\

&&0&2&1&1&1&1&0\\

&&0&29&1&1&0&0&0\\

&&0&29&1&1&1&1&0\\

&&0&58&1&1&0&0&0\\

&&0&58&1&1&1&1&1\\

&&1&1&1&1&0&0&0\\

&&1&2&1&1&0&0&0\\

&&1&29&1&1&0&0&0\\

&&1&58&1&1&0&0&1\\

&&2&1&1&1&0&0&0\\

&&2&2&1&1&0&0&29\\

&&3&1&1&1&0&0&30798\\

&&3&2&1&1&0&0&29\\

&&4&1&1&1&0&0&431172\\

&&4&2&1&1&0&0&406\\

&&5&1&1&1&0&0&4982432\\

&&5&2&1&1&0&0&406\\

&&6&1&1&1&0&0&44841888\\

&&6&2&1&1&0&0&3654\\

&&7&1&1&1&0&0&340933860\\

&&7&2&1&1&0&0&3654\\

&&8&1&1&1&0&0&2216070090\\

&&8&2&1&1&0&0&23751\\

&&9&1&1&1&0&0&12564269952\\

&&9&2&1&1&0&0&23751\\

&&10&1&1&1&0&0&62821349760\\

&&10&2&1&1&0&0&118755\\

&&11&1&1&1&0&0&279864762450\\

&&11&2&1&1&0&0&118755\\

&&12&1&1&1&0&0&1119459049800\\

&&12&2&1&1&0&0&475020\\

&&13&1&1&1&0&0&4047348325440\\

&&13&2&1&1&0&0&475020\\

&&14&1&1&1&0&0&13298430212160\\

&&14&2&1&1&0&0&1560780\\

&&15&1&1&1&0&0&39895474808520\\

&&15&2&1&1&0&0&1560780\\

&&16&1&1&1&0&0&109712555723430\\

&&16&2&1&1&0&0&4292145\\

&&17&1&1&1&0&0&277508616485760\\

&&17&2&1&1&0&0&4292145\\

\hline
\end{tabular}
\]
\end{table}
\setcounter{table}{0}
\begin{table}
\caption{Continued}
\[
\begin{tabular}{ll|llllll|l}
$p$ & $\alpha$ & $k$ & $d$ & $o_d(p)$ & $i$ & $j$  & $\beta$ & $\mathcal N$\\ \hline

59&1
&18&1&1&1&0&0&647520105133440\\

&&18&2&1&1&0&0&10015005\\

&&19&1&1&1&0&0&1397280911049870\\

&&19&2&1&1&0&0&10015005\\

&&20&1&1&1&0&0&2794561822099740\\

&&20&2&1&1&0&0&20030010\\

&&21&1&1&1&0&0&5189901539702880\\

&&21&2&1&1&0&0&20030010\\

&&22&1&1&1&0&0&8964375386759520\\

&&22&2&1&1&0&0&34597290\\

&&23&1&1&1&0&0&14420952951628860\\

&&23&2&1&1&0&0&34597290\\

&&24&1&1&1&0&0&21631429427443290\\

&&24&2&1&1&0&0&51895935\\

&&25&1&1&1&0&0&30284002423164672\\

&&25&2&1&1&0&0&51895935\\

&&26&1&1&1&0&0&39602157014907648\\

&&26&2&1&1&0&0&67863915\\

&&27&1&1&1&0&0&48402637241325122\\

&&27&2&1&1&0&0&67863915\\

&&28&1&1&1&0&0&55317299704371568\\

&&28&2&1&1&0&0&77558760\\

&&29&1&1&1&0&0&59132286206463754\\

&&29&2&1&1&0&0&77558760\\

&&29&29&1&1&0&0&2\\ \hline


61&1
&0&1&1&1&0&0&0\\

&&0&1&1&1&1&1&0\\

&&0&2&1&1&0&0&0\\

&&0&2&1&1&1&1&0\\

&&0&3&1&1&0&0&0\\

&&0&3&1&1&1&1&0\\

&&0&4&1&1&0&0&0\\

&&0&4&1&1&1&1&0\\

&&0&5&1&1&0&0&0\\

&&0&5&1&1&1&1&0\\

&&0&6&1&1&0&0&0\\

&&0&6&1&1&1&1&0\\

&&0&10&1&1&0&0&0\\

&&0&10&1&1&1&1&0\\

&&0&12&1&1&0&0&0\\

&&0&12&1&1&1&1&0\\

&&0&15&1&1&0&0&0\\

&&0&15&1&1&1&1&0\\

\hline
\end{tabular}
\]
\end{table}
\setcounter{table}{0}
\begin{table}
\caption{Continued}
\[
\begin{tabular}{ll|llllll|l}
$p$ & $\alpha$ & $k$ & $d$ & $o_d(p)$ & $i$ & $j$  & $\beta$ & $\mathcal N$\\ \hline

61&1
&0&20&1&1&0&0&0\\

&&0&20&1&1&1&1&0\\

&&0&30&1&1&0&0&0\\

&&0&30&1&1&1&1&0\\

&&0&60&1&1&0&0&0\\

&&0&60&1&1&1&1&1\\

&&1&1&1&1&0&0&0\\

&&1&2&1&1&0&0&0\\

&&1&3&1&1&0&0&0\\

&&1&4&1&1&0&0&0\\

&&1&5&1&1&0&0&0\\

&&1&6&1&1&0&0&0\\

&&1&10&1&1&0&0&0\\

&&1&12&1&1&0&0&0\\

&&1&15&1&1&0&0&0\\

&&1&20&1&1&0&0&0\\

&&1&30&1&1&0&0&0\\

&&1&60&1&1&0&0&1\\

&&2&1&1&1&0&0&0\\

&&2&2&1&1&0&0&30\\

&&3&1&1&1&0&0&32940\\

&&3&2&1&1&0&0&30\\

&&3&3&1&1&0&0&20\\

&&4&1&1&1&0&0&494100\\

&&4&2&1&1&0&0&420\\

&&4&3&1&1&0&0&20\\

&&4&4&1&1&0&0&15\\

&&5&1&1&1&0&0&5921880\\

&&5&2&1&1&0&0&420\\

&&5&4&1&1&0&0&15\\

&&5&5&1&1&0&0&12\\

&&6&1&1&1&0&0&55266000\\

&&6&2&1&1&0&0&4050\\

&&6&3&1&1&0&0&180\\

&&6&5&1&1&0&0&12\\

&&6&6&1&1&0&0&10\\

&&7&1&1&1&0&0&436012140\\

&&7&2&1&1&0&0&4050\\

&&7&3&1&1&0&0&180\\

&&7&6&1&1&0&0&10\\

&&8&1&1&1&0&0&2943156060\\

&&8&2&1&1&0&0&27300\\

&&8&4&1&1&0&0&105\\

\hline
\end{tabular}
\]
\end{table}
\setcounter{table}{0}
\begin{table}
\caption{Continued}
\vskip-10pt
\[
\begin{tabular}{ll|llllll|l}
$p$ & $\alpha$ & $k$ & $d$ & $o_d(p)$ & $i$ & $j$  & $\beta$ & $\mathcal N$\\ \hline

61&1
&9&1&1&1&0&0&17340022260\\

&&9&2&1&1&0&0&27300\\

&&9&3&1&1&0&0&1140\\

&&9&4&1&1&0&0&105\\

&&10&1&1&1&0&0&90168404160\\

&&10&2&1&1&0&0&142500\\

&&10&3&1&1&0&0&1140\\

&&10&5&1&1&0&0&60\\

&&10&10&1&1&0&0&6\\

&&11&1&1&1&0&0&418085456340\\

&&11&2&1&1&0&0&142500\\

&&11&5&1&1&0&0&60\\

&&11&10&1&1&0&0&6\\

&&12&1&1&1&0&0&1742022457200\\

&&12&2&1&1&0&0&593280\\

&&12&3&1&1&0&0&4800\\

&&12&4&1&1&0&0&450\\

&&12&6&1&1&0&0&40\\

&&12&12&1&1&0&0&5\\

&&13&1&1&1&0&0&6566185759500\\

&&13&2&1&1&0&0&593280\\

&&13&3&1&1&0&0&4800\\

&&13&4&1&1&0&0&450\\

&&13&6&1&1&0&0&40\\

&&13&12&1&1&0&0&5\\

&&14&1&1&1&0&0&22512637893600\\

&&14&2&1&1&0&0&2035800\\

&&15&1&1&1&0&0&70539862699800\\

&&15&2&1&1&0&0&2035800\\

&&15&3&1&1&0&0&15500\\

&&15&5&1&1&0&0&216\\

&&15&15&1&1&0&0&4\\

&&16&1&1&1&0&0&202802107059900\\

&&16&2&1&1&0&0&5851560\\

&&16&3&1&1&0&0&15500\\

&&16&4&1&1&0&0&1365\\

&&16&5&1&1&0&0&216\\

&&16&15&1&1&0&0&4\\

&&17&1&1&1&0&0&536829697508400\\

&&17&2&1&1&0&0&5851560\\

&&17&4&1&1&0&0&1365\\

&&18&1&1&1&0&0&1312250369330160\\

&&18&2&1&1&0&0&14307030\\

&&18&3&1&1&0&0&38640\\

&&18&6&1&1&0&0&120\\

\hline
\end{tabular}%
\]
\end{table}
\clearpage
\setcounter{table}{0}
\begin{table}
\caption{Continued}
\[
%
\]
\end{table}
\clearpage
\setcounter{table}{0}
\begin{table}
\caption{Continued}
\[
%
\]
\end{table}


%% file: aff-lin-gp.bbl
\begin{thebibliography}{99}

\bibitem{Agrell-Vardy-Zeger-IEEE-IT-2000}
E. Agrell, A. Vardy, K. Zeger, {\it Upper bounds for constant-weight codes}, IEEE Trans. Inform. Theory {\bf 46} (2000), 2373 -- 2395.

\bibitem{Bender-Goldman-AMM-1975}
E. A. Bender and J. R. Goldman, {\it On the applications M\"obius inversion in combinatorial analysis}, Amer. Math. Monthly {\bf 82} (1975), 789 -- 803. 


\bibitem{Brouwer-Etzion}
A. E. Brouwer and T. Etzion, {\it Some new distance-4 constant weight codes}, Adv. Math. Commun. {\bf 5} (2011), 417 -- 424.


\bibitem{Brouwer-Shearer-Sloane-Smith-IEEE-IT-1990}
A. E. Brouwer, J. B. Shearer, N. J. A. Sloane, W. D. Smith,  {\it A new table of constant weight codes}, IEEE Trans. Inform. Theory {\bf 36} (1990), 1334 -- 1380.

\bibitem{Cameron-Maimani-Omidi-Tayfeh-Rezaie-DM-2006}
P. J. Cameron, H. R. Maimani, G. R. Omidi, B. Tayfeh-Rezaie, {\it 3-Designs from $\text{\rm PSL}(2, q)$}, Discrete Math. {\bf 306} (2006), 3063 -- 3073.

\bibitem{Cameron-Omidi-Tayfeh-Rezaie-E-JC-2006}
P. J. Cameron, G. R. Omidi, B. Tayfeh-Rezaie, {\it 3-Designs from $\text{\rm PGL}(2, q)$}, Elec. J. Combin. {\bf 13} (2006), \#R50.


\bibitem{Chee-Xing-Yeo-IEEE-IT-2010}
Y. M. Chee, C. Xing, S. L. Yeo, {\it New constant-weight codes from propagation rules}, IEEE Trans. Inform. Theory {\bf 56} (2010), 1596 -- 1599.



\bibitem{Johnson-IEEE-IT-1962}
S. M. Johnson, {\it A new bound for error-correcting codes}, IEEE Trans. Inform. Theory, {\bf 8} (1962), 203 -- 207.

\bibitem{Hou-ppt} 
X. Hou, {\it The M\"obius function of the affine linear group $\text{\rm AGL}(1,\f_q)$}, preprint.

\bibitem{Huffman-Pless-2003}
W. C. Huffman and V. Pless, {\it Fundamentals of Error-Correcting Codes}, Cambridge University Press, Cambridge, UK, 2003.

\bibitem{Ostergard-IEEE-IT-2010}
P. R. J. \"Osterg\aa rd,  {\it Classification of binary constant weight codes}, IEEE Trans. Inform. Theory {\bf 56} (2010), 3779 -- 3785.


\bibitem{Schrijver-IEEE-IT-2005}
A. Schrijver, {\it New code upper bounds from the Terwilliger algebra and semidefinite programming}, IEEE Trans. Inform. Theory {\bf 51} (2005), 2859 -- 2866.


\bibitem{Semakov-Zinov'ev-PIT-1969}
N. V. Semakov and V. A. Zinov'ev, {\it Balanced codes and tactical configurations}, Problems of Information Transmission {\bf 5} (1969), 22 -- 28.

\bibitem{Smith-Hughes-Perkins-E-JC-2006}
D. H. Smith, L. A. Hughes, S. Perkins, {\it A new table of constant weight codes of length greater than 28}, Electron. J. Combin. {\bf 13} (2006),  Article 2, 18 pp.

\bibitem{Sun-TJM-2010}
H.-M. Sun, {\it From planar nearrings to generating blocks}, Taiwanese J. Math. {\bf 14} (2010), 1713 -- 1739.

\bibitem{Sun-JAA-2017}
H.-M. Sun, {\it Some sequences of optimal constant weight codes}, J. Alg. Appl. available online March 21, 2017.

\bibitem{Tonchev-1988} 
V. D. Tonchev, {\it Combinatorial Configurations}, Longman-Wiley, New York, 1988.


\end{thebibliography}
